\documentclass[preprint]{imsart}
\usepackage{amsthm,amsmath,amssymb,natbib}

\startlocaldefs

\numberwithin{equation}{section}
\theoremstyle{plain}
\newtheorem{theorem}{Theorem}[section]
\newtheorem{lemma}[theorem]{Lemma}
\newtheorem{proposition}[theorem]{Proposition}
\newtheorem{corollary}[theorem]{Corollary}
\newtheorem{assumption}[theorem]{Assumption}
\newtheorem{definition}[theorem]{Definition}
\newtheorem{remark}[theorem]{Remark}
\newtheorem{example}[theorem]{Example}

\DeclareMathOperator*{\esssup}{esssup}

\newcommand{\dd}{{\rm d}}
\newcommand{\DD}{\,{\rm d}}
\newcommand{\EE}{\mathbb{E}}

\newcommand{\HH}{\mathbb{H}}
\newcommand{\NN}{\mathbb{N}}
\newcommand{\PP}{\mathbb{P}}
\newcommand{\QQ}{\mathbb{Q}}
\newcommand{\RR}{\mathbb{R}}

\newcommand{\1}{\mathbf{1}}

\newcommand{\B}{\mathcal{B}}

\renewcommand{\H}{\mathcal{H}}
\newcommand{\Qorig}{Q}
\newcommand{\Qproxi}{{Q}^*}
\newcommand{\Qspec}[1]{{Q}^{(#1)}}
\newcommand{\tildeQproxi}{{\tilde Q}^*}
\newcommand{\tildeQspec}[1]{\tilde Q^{(#1)}}
\endlocaldefs

\begin{document}

\begin{frontmatter}

\title{On the Normalized Spectral Representation of Max-Stable Processes on a Compact Set}
\runtitle{Normalized Spectral Representation}


\begin{aug}
\author{\fnms{Marco} 
  \snm{Oesting}\thanksref{t1}\ead[label=e1]{oesting@uni-mannheim.de}},
\author{\fnms{Martin}
  \snm{Schlather}\thanksref{t1}\ead[label=e2]{schlather@uni-mannheim.de}}
\and
\author{\fnms{Chen}
  \snm{Zhou}\thanksref{t2}\ead[label=e3]{zhou@ese.eur.nl}}
\runauthor{M.~Oesting, M.~Schlather and C.~Zhou}
\affiliation{Universit\"at Mannheim, Universit\"at Mannheim,
             Erasmus University Rotterdam and De Nederlandsche Bank}
\address{M.~Oesting\\ M.~Schlather\\
        Institut f\"ur Mathematik \\Universit\"at Mannheim\\
        A 5, 6\\ 68161 Mannheim\\ Germany\\
        \printead{e1}\\
        \hphantom{E-mail: }\printead*{e2}}
\address{C.~Zhou\\
            Erasmus School of Economics\\ Erasmus University Rotterdam\\
	    P.O. Box 1738\\ 3000DR Rotterdam\\ The Netherlands\\
	    and\\
	    Economics and Research Division\\ De Nederlandsche Bank\\
            P.O. Box 98\\ 1000AB Amsterdam\\ The Netherlands\\
      \printead{e3}}
\end{aug}

\thankstext{t1}{Supported by Volkswagen Stiftung within the project `Mesoscale Weather Extremes -- Theory,
Spatial Modeling and Prediction (WEX-MOP)'.}

\thankstext{t2}{Views expressed do not necessarily reflect official positions of De Nederlandsche Bank.}

\begin{abstract}
The normalized spectral representation of a max-stable process on a compact set
is the unique representation where all spectral functions share the same
supremum. Among the class of equivalent spectral representations of a process,
the normalized spectral representation plays a distinctive role as a solution
of two optimization problems in the context of an efficient simulation of
max-stable processes. Our approach has the potential of considerably reducing
the simulation time of max-stable processes. 
\end{abstract}



\end{frontmatter}

\section{Introduction}

Max-stable processes have become a popular tool for modeling spatial extremes,
particularly in environmental sciences, see, e.g. \citet{coles93}, 
\citet{colestawn96} and \citet{padoan2010likelihood}.
 Let $Z=\{Z(y) : y \in K\}$ be a max-stable process with standard
 Fr\'{e}chet margins defined on an index set $K$. Then, there exists a spectral
 measure $H$ defined on an appropriate set of functions $\HH$ such that
\begin{equation} \label{eq:def}
  Z(y) = \max_{(t,f) \in \Pi} t f(y), \qquad y\in K,
\end{equation}
where $\Pi$ is the Poisson point process on $(0,\infty) \times \HH$
with intensity measure $t^{-2} \DD t \, H(\dd f)$ and
\begin{equation} \label{eq:std-frechet}
 \int_{\HH} f(y) \, H(\dd f) =1
\end{equation}
for all $y \in K$, see \cite{haan84,GHV90,kabluchko09d} and \cite{wangstoev10},
for instance. The non-negative shape functions $f$ in $\HH$ are the 
{\it spectral functions} that correspond to the max-stable process $Z$. 

The ensemble of spectral functions corresponding to a given max-stable process
is not unique \citep[cf.][Remark 9.6.2]{DHF} and a choice has to be made in
applications. Some specific choices may bear severe disadvantages.
For instance, finite approximations based on the original definition of the
Brown-Resnick process are far from the actual process, in general 
\citep{KSH09}. Nonetheless, the optimality of the choice of spectral functions
has not been discussed in literature yet. Here, we propose a criterion for
choosing spectral functions that is the solution to an optimization problem
stemming from unconditional simulation of max-stable processes.

From both a theoretical and a practical point of view, it is important 
to be able to draw random samples from a max-stable process. 
While bivariate marginal distributions can be calculated frequently, higher
dimensional marginal distributions do not have, in nearly all the cases,
explicit formulae. Consequently, they can be addressed only by 
simulation.  Furthermore, most applications require the estimation of
characteristics of max-stable processes that cannot be explicitly
calculated. That leaves simulation as the only option, see, e.g.\ 
\citet{buishand2008spatial} and \citet{blanchet2011spatial}.
Finally, unconditional simulation appears as part of the conditional 
simulation of max-stable processes \citep{DEMR13,oestingschlather12}. 

According to the spectral representation \eqref{eq:def}, the construction of a
max-stable process involves infinitely many points $(t,f)\in \Pi$. 
Nevertheless, since only the maximum over all functions $tf$ counts, the number
of points $(t,f)$ that contribute to $Z$, i.e. $Z(y) = t f(y)$ for at least one
point $y\in K$, is finite under mild conditions, see \cite{DHF}, Cor.\ 9.4.4.
However, their statement is a theoretical one that does not help for simulation
purposes because one cannot determine ex ante which function $f$ will 
contribute. Assuming that $H$ is finite, \cite{schlather02} suggests to start
with those points $(t,f)$ that will contribute most likely to $Z$, i.e., with
those that have the highest values of $t$. By ranking the points $t$ in a
descending order $t_1> t_2,\ldots$ and assuming without loss of generality that
$H$ is a probability measure, we have that $t_i=^d1/(\sum_{j=1}^i E_j)$,
where $E_j$ are independent and identically distributed random variables with
standard exponential distribution. Let $f_i\sim_{i.i.d.} H$ be independent of
the $E_j$ and
\begin{equation} \label{eq:m}
  Z^{(m)}(y) = \max_{1\leq i\leq m} \frac{1}{\sum_{j=1}^i E_j} f_i(y),
\qquad y\in K, 
\end{equation}
a finite approximation for $Z$. Then, $Z =^d Z^{(\infty)}$, i.e.
\begin{equation*}
Z(y) =^d \max_{i\geq 1} \frac{1}{\sum_{j=1}^i E_j} f_i(y), \qquad y\in K.
\end{equation*}
Therefore, if for a given $m$ we have that 
\begin{equation} \label{eq:stopping}
Z^{(m)}(y)\geq \frac{1}{\sum_{j=1}^m E_j}\sup_{f\in \HH} f(y) \quad
\text{ for all } y \in K,
\end{equation}
then, obviously, $Z^{(n)}(y) = Z^{(\infty)}(y)$ for all $y\in K$ and all 
$n\ge m$. In other words, any spectral function $f_i$ with $i> m$ cannot
contribute to $Z$. This results in a {\it stopping rule} for a ``$m$-step
representation'' of $Z$, where $m$ is a random integer. Such a stopping rule
can be applied to construct an exact simulation algorithm. In the case of 
Brown-Resnick processes, \citet{OKS12} compare this algorithm to algorithms
based on other representations. The results of \cite{schlather02} imply that
$m$ is finite almost surely if, for instance, $K$ is finite, the shape
functions are uniformly bounded and their support is included in a fixed
compact set. As a side result, we shall show that even the expectation of $m$
is finite, under rather mild conditions.

We present a toy example to clarify why the choice of spectral functions
can have a major impact on the distribution of the stochastic number $m$.
Consider the simplest case where $Z$ is univariate. Specializing 
\eqref{eq:def} to $K=\{y_0\}$ and $f\equiv 1$, the random variable $Z(y_0)$
follows a  univariate Fr\'echet distribution. It has a
representation given by
\begin{equation} \label{eq:def2}
  Z(y_0) =^d \max_{t \in \Pi} t 
\end{equation}
where $\Pi$ is the Poisson point process on $(0, \infty)$ with intensity
$t^{-2} \DD t$. Obviously, the right-hand side of \eqref{eq:def2}
is fully given by the largest value of $t\in\Pi$. In other words, $m\equiv 1$.
Now, let us consider the general case: $Z(y_0)$ is given by \eqref{eq:def}
and $f(y_0)$ is a non-degenerate random variable with expectation 
$\EE f(y_0)=1$. Then, the stochastic number $m$ is greater than $1$
 with positive probability. Even worse, if the right endpoint of
 $f(y_0)$ is infinite then $m=\infty$ almost surely.
In practical applications, in particular for simulating $Z(y_0)$, the
spectral representation in \eqref{eq:def2} would be considered as optimal.
This example illustrates the optimality we intend to achieve by the
choice of spectral functions for an arbitrary max-stable process.

The very general optimality problem for general index sets $K$ and arbitrary
random functions $f$ seems to be rather complicated. Therefore, we shall
suggest a modified optimization problem and shall demonstrate that its solution
is explicit and unique for each given max-stable process and index set $K$. It
can be achieved via resealing any ensemble of spectral functions to a new 
ensemble of spectral functions satisfying $\sup_{y\in K} f(y)=c$, for all
$f \in \HH$.  We call such a representation with all spectral functions
sharing the same supremum the \emph{normalized spectral representation}.
This representation was initially used in constructing the spectral
representation for sample-continuous max-stable processes on $K=[0,1]$, see
e.g.\ \cite{dehaanlin01} and \cite{DHF}, Cor.\ 9.4.5. Hence, in this paper, we
give a theoretical  justification on the optimality of the normalized spectral
representation.

This paper is organized as follows. In Section \ref{sec:background}, we revisit
de Haan's \citeyearpar{haan84} spectral representation of max-stable processes
and give a formula how to transform one ensemble of spectral functions under a
given spectral measure to another ensemble under a different spectral measure.
We focus on a particular transformation leading to the normalized spectral
representation. We state necessary and sufficient conditions on the existence
and show the uniqueness of this representation. In Section \ref{sec:optim} we
define the optimization problem and give the explicit solution of the
replacement problem, the normalized spectral representation. The replacement
problem is evaluated and refined in Section \ref{sec:refine}. Section
\ref{sec:examples} deals with examples of the normalized spectral 
representation for specific cases of the max-stable process as well as the
index set $K$. In Section \ref{sec:simu}, for Smith's \citeyearpar{smith90}
process, the number $m$ of considered spectral functions in the normalized
spectral representation is compared to the corresponding number in the
algorithm proposed by \citet{schlather02} in a simulation study. The paper
closes with a summary and discussion of our results.

\section{The normalized spectral representation} \label{sec:background}

Throughout the paper we assume that the index set $K$ is a compact
Polish space. The following proposition shows how to transform one spectral
representation to another one.
\begin{proposition}\label{prop:1}
 Let $Z$ be a max-stable process with standard Fr\'echet margins defined
 as in \eqref{eq:def} and \eqref{eq:std-frechet} where the spectral functions
 $f$ are in some Polish space $\HH \subset [0,\infty)^K$.
  
 Suppose $H$ is a locally finite measure on $\HH$. 
 Let $g$ be some probability density on $\HH$ w.r.t.\ $H$,
 i.e.\ $g \geq 0$ and $\int_{\HH} g(f) H(\dd f) = 1$, such that
 \begin{equation} \label{eq:g-regularity}
   H\left( \left\{f: g(f) =0, \ \sup_{y \in K} f(y) > 0\right\}\right) = 0.
 \end{equation}
 Then,
 \begin{eqnarray} \label{eq:0}
  Z(y) =^d \max_{(t,f) \in \tilde \Pi} t \frac{f(y)}{g(f)}, \qquad y\in K,  
 \end{eqnarray}
 where $\tilde \Pi$ is a Poisson point process with intensity
 $t^{-2} \DD t \, g(f) H(\dd f)$.
\end{proposition}
\begin{proof}
For any finite subset $\{y_i: i \in I\} \subset K$ and $z_i > 0$, $i \in I$,
we have
\begin{align*}
   \PP(Z(y_i) \le z_i, \ i \in I) 
  ={}& \PP\left(\left|\Pi \cap \{(t,f) : \ tf(y_i) > z_i \hbox{ for some } i\in I\}\right|=0\right)\\
  ={}& \exp\left(
    -\int_{\HH} \int_{\min_{i \in I} \{ z_i / f(y_i)\}}^\infty
    t^{-2} \DD t \, H(\dd f)
    \right) \displaybreak[0]\\
 ={} & \exp\left(
    -\int_{\HH} \int_{\min_{i \in I} \{ z_i / [f(y_i) / g(f)]\}}^\infty
    t^{-2} \DD t \, g(f) \, H(\dd f)
    \right)\\
 ={} & \PP\left( \max_{(t,f) \in \tilde \Pi} t \frac{f(y_i)}{g(f)} \leq z_i,
                 \ i \in I \right).
\end{align*}
\end{proof}

Applying Proposition \ref{prop:1}, one can transform the given set of spectral
functions $\{f\}_{(t,f) \in \Pi}$ to a new set 
$\{f/g(f)\}_{(t,f) \in \tilde \Pi}$, where $f$ follows the transformed
probability measure $gH$ defined by $$\textstyle gH(A) = \int_A g(f) H(\dd f)$$
for all measurable sets $A \subset \HH$. We will focus on a particular choice
of $g$ which leads to the normalized spectral representation as follows. Let
$f \mapsto \sup_{y \in K} f(y)$ be measurable on $\HH$ and assume that
$$ \textstyle c: = \int_{\HH} \sup_{y \in K} f(y) \, H(\dd f) < \infty.$$
Then, the choice $g = g^*$ defined by
\begin{equation} \label{eq:g-star}
 g^*(f) := \textstyle c^{-1} \sup_{y \in K} f(y), \quad f \in \HH,
\end{equation}
satisfies the assumptions of Proposition \ref{prop:1}. Therefore, $Z =_d \tilde Z$ for
\begin{equation} \label{eq:z-tilde}
\tilde Z(y) = \max_{t \in \Pi_0} t \frac{c F_t(y)}{\sup_{\tilde y \in K} F_t(\tilde y)}, \quad y \in K,
\end{equation}
where $\Pi_0$ is a Poisson point process on $(0,\infty)$ with intensity 
$t^{-2} \dd t$ and $F_t$, $t>0$, are independent random processes with density
$c^{-1} \sup_{y \in K} f(y) \, H(\dd f)$. The modified spectral functions
$\{c F_t/\sup_{y \in K} F_t(y)\}$ can be perceived as independent copies
of a stochastic process $F^*$ with 
\begin{equation} \label{eq:sup-const}
\textstyle \sup_{y \in K} F^*(y) \equiv c.
\end{equation}

\begin{definition} \label{def:spec-repr}
Let $Z$ be a max-stable process on $K$ satisfying
\begin{equation} \label{eq:spec-repr}
 Z =^d \max_{t \in \Pi_0} t F^*_t. 
\end{equation}
Here,  $\Pi_0$ is a Poisson point process on $(0,\infty)$ with intensity 
$t^{-2} \dd t$ and $F^*_t$, $t>0$, are independent copies of a stochastic
process $F^*$ satisfying \eqref{eq:sup-const} for some $c \in (0,\infty)$.
Then, the right-hand side of \eqref{eq:spec-repr} is called normalized
spectral representation of $Z$.
\end{definition}

The choice $g = g^*$ in Proposition \ref{prop:1} leads to a valid normalized
spectral representation only if
$c = \int_{\HH} \sup_{y \in K} f(y) \, H(\dd f) < \infty$.
Note that, in general, $c$ is not necessarily finite even though we assume
$\int_\HH f(y) H(\dd f) = 1$ for all $y \in K$. However, $c$ is finite whenever
$K$ consists of a finite number of points. The following proposition deals with
equivalent conditions for $c < \infty$ in a more general setting, replacing 
$f \mapsto \sup_{y \in K} f(y)$ by an arbitrary max-linear functional.

\begin{proposition} \label{prop:c-finite}
 Assume that we are in the framework of Proposition \ref{prop:1}. 
 Furthermore, assume that the function $L: \HH \to (0,\infty)$
 is max-linear and measurable.
 Then the following conditions are equivalent:
 \begin{enumerate}
  \item[1. ] $c_L := \int_\HH L(f) \, H(\dd f) < \infty$
  \medskip
  
  \item[2. ] $\PP(L(Z) \leq a) > 0$ for some $a > 0$  
  \medskip
  
  \item[3. ] $\PP(L(Z) < \infty) = 1$ 
  (or, equivalently, $\PP(L(Z) < \infty) > 0$).  
 \end{enumerate}
 If we additionally assume that there is some stochastic process $W$ such that
 \begin{equation} \label{eq:incremental}
  Z =^d \max_{t\in\Pi_0} t W_t,
 \end{equation}
 where $\Pi_0$ is a Poisson point process on $(0,\infty)$ with intensity 
 $t^{-2} \DD t$ and $W_t$, $t>0$, are independent copies of $W$,
 we get another equivalent condition:
 \begin{enumerate}
  \item[4. ] $\EE L(W) < \infty.$
 \end{enumerate}
\end{proposition}
\begin{proof}
 The assertion follows from the following continued equality:
 \begin{align*}
 \exp\left(- \frac {c_L} a \right) ={} & \exp\left(-\int_\HH \int^\infty_{a/L(f)} t^{-2} \DD t \,H(\dd f) \right)
            {}={}  \PP\left(L(Z) \leq a\right)\\
            ={} & \exp\left(-\EE_W\left( \int^\infty_{a/L(W)} u^{-2} \DD u \right) \right) 
          {}={} \exp\left( - a^{-1} \EE L(W) \right).
 \end{align*}
 for any $a >0$. The equivalence to the third assertion follows from the relation
 $ \PP(L(Z) < \infty) = \lim_{a \to \infty} \PP(L(Z) \leq a)$.
\end{proof}

\begin{remark}
 Similar results, presenting equivalent statements for some special choices of 
 $L$, can already be found in the literature; see, for instance, 
 \cite{resnickroy91}, who showed the equivalence of the first and third 
 assertion for $L(f) = \sup_{y \in K} f(y)$. For this choice of $L$, it 
 follows that $c$ is finite if $Z$ has continuous sample paths.
\end{remark}

While Proposition \ref{prop:c-finite} is related to the question of the
existence of a normalized spectral representation, the following proposition
deals with its uniqueness.

\begin{proposition} \label{prop:uniqueness}
 Let $Z$ be a max-stable process with a normalized spectral representation.
 Furthermore, let $Z^K := \sup_{y \in K} Z(y)$.
 
 Then, we have
 \begin{enumerate}
  \item[1.] $c = - \log \PP\left(Z^K \leq 1\right)$
  \item[2.] For any $y_1, \ldots, y_n \in K$, $w_1, \ldots, w_n > 0$, it holds
  \begin{align}
   & \PP(F^*(t_i) \leq w_i, \, 1 \leq i \leq n) \nonumber\\
   ={} & \lim_{z \to \infty} \PP\left(\frac{Z(y_i)}{Z^K} \leq \frac{w_i} c, \, 1 \leq i \leq n \, \bigg| \, Z^K > z\right). \label{eq:w-aslim}
  \end{align}
 \end{enumerate}
\end{proposition}
\begin{proof}
 The first part is a consequence of the proof of Proposition 
 \ref{prop:c-finite}. For the proof of the second part, let
 $\tilde \Pi = \{(t,F^*_t): \ t \in \Pi_0\}$. Then, we have
 \begin{align}
  & \PP\bigg(\left|\tilde \Pi \cap \left\{(u,w): \, u > \frac z c\right\}\right| > 0, \nonumber \\
  & \phantom{\PP\bigg(} \left|\tilde \Pi \cap \left\{(u,w): \, u > \frac z c, \ 1 > \min_{1 \leq i \leq n} \frac{w_i}{w(y_i)}\right\}\right| =0, \nonumber \\
  & \phantom{\PP\bigg(} \left|\tilde \Pi \cap \left\{(u,w): \, u \leq \frac z c, \,
                      \frac u z > \min_{1 \leq i \leq n} \frac{w_i}{c w(y_i)}\right\}\right| =0\bigg)\nonumber \displaybreak[0]\\
  \leq{}&  \PP\bigg(\frac{Z(y_i)}{Z^K} \leq \frac{w_i} c, \, 1 \leq i \leq n, \, Z^K > z\bigg) \nonumber
                                     \displaybreak[0]\\
  \leq{}& \PP\bigg( \left|\tilde \Pi \cap \left\{(u,w): \, u > \frac z c, \, \max_{1 \leq i \leq n} \frac{w(y_i)}{w_i} \leq 1\right\}\right| > 0\bigg).\label{eq:sandwich}
 \end{align}
 The lower bound in \eqref{eq:sandwich} equals
 \begin{align*}
   & \left(1 - \exp\left(-\frac c z \PP(F^*(y_i) \leq w_i, \, 1 \leq i \leq n)\right)\right)\\ 
   & \cdot \exp\left(-\frac c z \PP\left(\max_{1 \leq i \leq n} \frac {F^*(y_i)}{w_i} > 1\right)\right)
   \exp\left( - \EE \int_{\frac z c  \wedge \min_{1 \leq i \leq n}  \frac{zw_i}{cF^*(y_i)} }^{\frac z c} u^{-2} \dd u\right)\\
 ={}& \left(1 - \exp\left(-\frac c z \PP(F^*(y_i) \leq w_i, \, 1 \leq i \leq n)\right)\right)\\ 
   & \cdot \exp\left(-\frac c z \PP\left(\max_{1 \leq i \leq n} \frac {F^*(y_i)}{w_i} > 1\right)\right)
     \exp\left(- \frac c z \EE \left( \max_{1 \leq i \leq n} \frac{F^*(y_i)}{w_i} -1\right)_{+}\right),
 \end{align*}
 while the upper bound equals
 \begin{align*}
  1 - \exp\left(-\frac c z \PP(F^*(y_i) \leq w_i, \, 1 \leq i \leq n)\right).
 \end{align*}
 Using $\PP(Z^K > z) = 1 - e^{-\frac c z}$ and
 taking the limit $z \to \infty$, inequation \eqref{eq:sandwich} yields
 \eqref{eq:w-aslim}.
\end{proof}

\begin{remark}
 Proposition \ref{prop:uniqueness} shows that the law of $F^*$ is uniquely
 determined on the $\sigma$-algebra generated by the cylinder sets
 $$ \{f \in [0,\infty)^K: \ f(t_i) \in B_i, \ 1 \leq i \leq m\},$$
 with $t_1, \ldots, t_m \in K$, $B_1, \ldots, B_m \in \mathcal{B} \cap [0,\infty),$ 
 and $m \in \NN$. If we restrict ourselves to continuous functions, the
 corresponding trace $\sigma$-algebra is the Borel $\sigma$-algebra. Thus,
 the normalized spectral representation is unique for sample-continuous
 processes. Uniqueness also holds true for some more general classes of
 processes, e.g.\ c\`adl\`ag processes on $\RR$.
\end{remark}

The following statement on the existence and uniqueness of the normalized
spectral representation follows directly from Propositions \ref{prop:c-finite}
and \ref{prop:uniqueness}.

\begin{corollary} \label{coro:spec-repr}
 Let $Z$ be a max-stable process as defined in \eqref{eq:def} such that
 $f \mapsto \sup_{y \in K} f(y)$ is measurable. Then,  $Z$ allows for a
 normalized spectral representation if and only if 
 $\sup_{y \in K} Z(y) < \infty$ a.s. In this case, the spectral process $F^*$
 is unique in the sense of finite-dimensional distributions.
\end{corollary}

\begin{remark}
  Note that the measurability of $f \mapsto \sup_{y \in K} f(y)$
  is ensured if $K$ is countable or if every spectral function $f \in \HH$ is
  upper semi-continuous and its subgraph
  $$ U(f) = \{ (y,z) \in K \times [0,\infty): \, f(y) \leq z\}$$
  satisfies $\overline{U(f)^\circ} = U(f)$, i.e.\ $U(f)$ equals the closure of
  its interior.  
\end{remark}

The normalized spectral representation is particularly useful in the
context of ``$m$-step representations'' as in \eqref{eq:stopping}. Recall
that, in the representation $Z = \max_{(t,f) \in \Pi} t f$, usually only few
points in the Poisson point process $\Pi$ contribute to $Z$ as pointed out in
the introduction. And the points $(t,f) \in \Pi$ satisfying 
$t < Z(y) / \sup_{f \in \HH} f(y)$ are not able to contribute to $Z(y)$,
$y \in K$. This statement also holds for the transformed set of spectral
functions $\{f/g(f)\}$ as constructed in Proposition \ref{prop:1}. Here,
points $(t,f) \in \tilde \Pi$ are not able to contribute if   
\begin{equation} \label{eq:stopping-g}
t < Z(y) / \sup_{f \in \HH} (g(f)^{-1} f(y)).
\end{equation}
In case of the normalized spectral representation, this implies that points
$(t,f) \in \tilde \Pi$ cannot contribute if $t < Z(y)/c$. In particularly,
the number of points which are able to contribute to $Z$ is finite a.s.\
provided that $\inf_{y \in K} Z(y) > 0$ with probability one. If the normalized
spectral representation does not exist, then, by Corollary \ref{coro:spec-repr},
$\PP(\sup_{y \in K} Z(y) = \infty) > 0$ (or, equivalently, $c=\infty$)
which makes the existence of an ``$m$-step representation'' doubtful. 
Later, we will show that the expected number of points which are able to
contribute to $Z$ is infinite for any spectral representation of type
\eqref{eq:0} in this case (see Remark \ref{rem:inf-number}). In general, 
however, this number depends on the choice of spectral functions as 
\eqref{eq:stopping-g} suggests. This observation gives rise to the question of
an optimal choice of spectral functions such that the number of points that are
able to contribute to $Z$ is minimized. We formulate this problem in the next
section and show that the normalized spectral representation solves a modified
version of the optimization problem.

\section{The optimization problem} \label{sec:optim}

In the following, we will always assume that we are in the framework of
Proposition \ref{prop:1}. Further, the process $Z$ is assumed to be almost
surely strictly positive on $K$, i.e.\
\begin{equation} \label{eq:Z-pos}
  \PP\left(\inf_{y \in K} Z(y) > 0\right) = 1.
\end{equation}

We are interested in minimizing the number of considered spectral functions. 
For the transformed spectral representation \eqref{eq:0} with spectral
functions $\{f/g(f)\}_{(t,f) \in \tilde \Pi}$, the stopping rule
\eqref{eq:stopping} can be formulated as follows. Denote
\begin{equation} \label{eq:z-m}
Z^{(m)}(y) = \max_{1 \leq i \leq m} \frac1{\sum_{j=1}^i E_j}\cdot \frac{f_i(y)}{g(f_i)},\qquad y\in K,
\end{equation}
for standard exponentially distributed random variables $E_j$ and $f_j\sim gH$, 
which are all independent. Let 
\begin{equation} \label{eq:z-infty}
 Z^{(\infty)} = \lim_{m\rightarrow \infty} Z^{(m)}.
\end{equation}
Then, $Z^{(\infty)} =^d Z$ and, for fixed $\omega \in \Omega$, we have 
$Z^{(m)} \equiv Z^{(\infty)}$ on $K$ if
\begin{equation} \label{eq:stopping-trafo}
 \esssup_{f \in \HH} \sup_{y \in K} \frac{f(y)}{g(f) Z^{(m)}(y)} \le \sum_{j=1}^{m+1} E_j,
\end{equation}
where the essential supremum is taken w.r.t.\ the probability measure $gH$.
Thus, for fixed $g$, we may exclude all the functions $f \in \HH$ with $g(f)=0$.
By \eqref{eq:g-regularity}, up to a set of $H$-measure zero, the set 
$\{f \in \HH: \ g(f)=0\}$ consists of functions $f \in \HH$ with 
$f\mid_K \equiv 0$.

For a choice of spectral functions that minimizes the number $m$ such that
\eqref{eq:stopping-trafo} holds, we need to determine at least one member of
\begin{equation*} 
\Qorig ={}  \arg\min_{g} \Qorig_{g}
\end{equation*}
where
\begin{align} \label{eq:original}
\Qorig_{g} ={} & \EE \min\bigg\{ m\in \NN : \esssup_{f \in \HH} \sup_{y\in K}
\frac{f(y)}{g(f) Z^{(m)}(y)} \le \sum_{j=1}^{m+1} E_j\bigg\}.
\end{align}

\begin{remark} \label{rem:infz-finite}
 By \eqref{eq:Z-pos}, we assume that $\inf_{y \in K} Z(y) > 0$ almost surely.
 If we additionally assume that 
 $\esssup_{f \in \HH} \sup_{y \in K} \frac{f(y)}{g(f)} < \infty,$
 this guarantees that, with probability one, the stopping rule
 \eqref{eq:stopping-trafo} holds for some finite $m$.

 Clearly, \eqref{eq:Z-pos} is satisfied if $Z$ is sample-continuous.
 Note that a much weaker assumption than sample-continuity already implies
 a much stronger statement than \eqref{eq:Z-pos}, namely
 \begin{equation} \label{eq:infz-finite}
   \EE \bigg[ \bigg(\inf_{y \in K} Z(y)\bigg)^{-1}\bigg] < \infty.
 \end{equation}
 For \eqref{eq:infz-finite} to hold, it suffices that, for every $y \in K$, 
 there exist an open set $U(y)$ containing  $y$, a set of spectral functions
 $M(y) \subset \HH$ with $H(M(y)) > 0$ and a real number $a(y) >0$ such that
 $f(x) > a(y)$ for all $f \in M(y)$ and $x \in U(y)$. 
 
 To see this, we first note that a finite set $Y = \{y_1, \ldots, y_n\} \subset K$
 exists such that $\bigcup_{i=1}^n U(y_i) \supset K$ as $K$ is compact. Without
 loss of generality, we may assume that the corresponding sets 
 $M(y_1), \ldots, M(y_n)$ are pairwise disjoint. Otherwise, i.e.\ if 
 $H(M(y_i) \cap M(y_j)) > 0$ for some $j \neq i$, the indices $i$ and $j$ can
 be merged by considering $U(y_i) \cup U(y_j)$ instead of $U(y_i)$, 
 $M(y_i) \cap M(y_j)$ instead of $M(y_i)$, $\min\{a(y_i), a(y_j)\}$ instead of
 $a(y_i)$ and $Y \setminus \{y_j\}$ instead of $Y$.
 Now, for any $z > 0$, we have that
 \begin{align*}
   \PP\left( \inf_{y \in K} Z(y) \geq z\right)
 \geq{} & \PP\left( | \Pi \cap (z/a(y_i),\infty) \times M(y_i)| > 0,\, 1 \leq i \leq n\right) \displaybreak[0]\\
 ={}  \prod_{i=1}^n \bigg(&1 - \exp\left( - \frac{H(M(y_i)) a(y_i)} z\right)\bigg)
 \geq{} \left(1 - \exp\left(- \frac {p}{z}\right)\right)^n,
 \end{align*}
 where $p = \min_{1 \leq i \leq n} H(M(y_i)) a(y_i) > 0$.
 By Bernoulli's inequality, we obtain
 $$\PP\left(\inf_{y \in K} Z(y) \geq z\right) \geq 1 - n \exp\left(- \frac{p} z \right)$$
 or, equivalently,
 $$\textstyle \PP\left( \left( \inf_{y \in K} Z(y)\right)^{-1} > z\right) \leq n \exp\left(- p z\right).$$
 Thus,
 $$ \textstyle \EE \left[ \left(\inf_{y \in K} Z(y)\right)^{-1}\right] \leq \int_0^\infty n \exp\left(- p z\right) \dd z < \infty.$$
\end{remark}

The optimization problem \eqref{eq:original} is difficult to solve because
both the numerator and the denominator of $$\frac{f(y)}{g(f) Z^{(m)}(y)}$$
depend on $y$ and the denominator is stochastic.
Hence, we consider some modified versions of the optimization problem whose
solutions are expected to be rather close to that of the actual problem.

\subsection{A modified optimization problem} \label{subsec:simple}

Recall that the stopping rule \eqref{eq:stopping-trafo} requires that
$$ \esssup_{f \in \HH} \sup_{y \in K} \frac{f(y)}{g(f) Z^{(m)}(y)} \leq \sum_{j=1}^{m+1} E_j.$$
A stronger condition than \eqref{eq:stopping-trafo} is then
\begin{equation} \label{eq:cond-strong}
 \esssup_{f \in \HH} \frac{\sup_{y \in K} f(y)}{g(f) \inf_{\tilde y \in K} Z^{(m)}(\tilde y)} \leq \sum_{j=1}^{m+1} E_j,
\end{equation}
while a weaker condition is
\begin{equation} \label{eq:cond-weak}
 \esssup_{f \in \HH} \frac{\sup_{y \in K} f(y)}{g(f) \sup_{\tilde y \in K} Z^{(m)}(\tilde y)} \leq \sum_{j=1}^{m+1} E_j.
\end{equation}
The actual stopping rule is in between the strong and the weak condition.
Suppose $T: [0,\infty)^K \to [0,\infty)$ is a functional that satisfies 
$T(\1) = 1$ and that is max-linear, i.e.\
\begin{equation*}
 \quad T(\max\{a_1 h_1, a_2 h_2\}) = \max \{ a_1 T(h_1), a_2 T(h_2) \}
\end{equation*}
for all $a_1, a_2 > 0$ and $h_1, h_2: K \to [0,\infty)$.
Then, we have that $T(h) \le T(g)$ for all $h\le g$, which leads to
\begin{align} \label{eq:F-bounds}
 \textstyle \inf_{y \in K} h(y) \leq T(h) \leq \sup_{y \in K} h(y)
\end{align}
for all $h: K \to [0,\infty)$.
We consider the condition
\begin{equation} \label{eq:cond-maxlin}
   \esssup_{f \in \HH} \frac{\sup_{y \in K} f(y)}{T(Z^{(m)}) g(f)} \leq \sum_{j=1}^{m+1} E_j 
\end{equation}
for some suitable $T$, and regard the new condition \eqref{eq:cond-maxlin}
as a proxy for the actual stopping rule \eqref{eq:stopping-trafo}.
Apparently, condition \eqref{eq:cond-maxlin} 
lies in between the strong condition
\eqref{eq:cond-strong} and the weak condition \eqref{eq:cond-weak}.
The corresponding modified optimization problem is then
\begin{align}
\Qproxi {}={}& \arg\min_{g} \Qproxi_{g}, \nonumber\\
\Qproxi_{g} {}={}& \EE \min\bigg\{ m\in \NN : \esssup_{f\in\HH}
\frac{\sup_{y \in K} f(y)}{g(f) T(Z^{(m)})} 
\le \sum_{j=1}^{m+1} E_j \bigg\}.   \label{eq:simple}
\end{align}
In fact, Proposition \ref{prop:simplification} below
shows that the solution of the modified problem in 
\eqref{eq:simple} is not related to the particular choice of the
max-linear functional $T$. Therefore, we regard the solution of the
modified optimization problem in \eqref{eq:simple} as a good proxy to
that of the original problem in \eqref{eq:original}.

Examples of $T$ are $T(h) = \sup_{y \in K} h(y)$ and $T(h) = h(y_0)$
for some $y_0 \in K$. Thus, we get that minimizing
\begin{align}
\Qspec1_{g}& {}={} \EE \min\bigg\{ m\in \NN: \esssup_{f\in\HH}
\frac{\sup_{y \in K} f(y)} {g(f) \sup_{\tilde y \in K} Z^{(m)}(\tilde y)}
\le \sum_{j=1}^{m+1} E_j \bigg\}, \label{eq:simple-1}\\
\text{or} \qquad & \nonumber\\
\Qspec2_{g}&(y_0) {}={} \EE \min\bigg\{ m\in \NN: \esssup_{f\in\HH} 
\frac{\sup_{y \in K} f(y)}{g(f) Z^{(m)}(y_0)}
\le \sum_{j=1}^{m+1} E_j \bigg\}, \label{eq:simple-2}
\end{align}
are important modifications of the original optimization problem.

\subsection{The solution of the modified optimization problem}

The following proposition provides a first step to the solution
of the modified optimization problem.

\begin{proposition}\label{prop:simplification}
Let $f \mapsto \sup_{y \in K} f(y)$ be measurable. Then, 
\begin{equation*}
\Qproxi = \arg\min_{g} \esssup_{f\in\HH} \frac{\sup_{y \in K} f(y)}{g(f)}.
\end{equation*}
In particular, $\Qproxi$ does not depend on the choice of $T$.
\end{proposition}

\begin{proof}
  If there exists some $g$ such that $\Qproxi_{g}$ is finite, then necessarily
  $$\esssup_{f \in H} \frac{\sup_{y \in K} f(y)}{g(f)} < \infty.$$
  Thus, we can restrict ourselves to 
  $$ g \in D = \bigg\{g: \,\esssup_{f\in\HH} \frac{\sup_{y \in K} f(y)}{g(f)} < \infty\bigg\}$$
  and assume w.l.o.g.\ that $D \neq \emptyset$.
  For $c = \int_{\HH} \sup_{y \in K} f(y) \, H(\dd f)$ and any
  $g \in D$, we have
  \begin{align}
  c \leq{} & \int_{\HH} \esssup_{h \in \HH} \frac{\sup_{y \in K} h(y)}{g(h)} g(f) \, H(\dd f) 
  {}={} \esssup_{h \in \HH} \frac{\sup_{y \in K} h(y)}{g(h)} < \infty. \label{eq:c-bound}
  \end{align}
  Thus, by \eqref{eq:F-bounds}, for $c_T =  \int_{\HH} T(f) \, H(\dd f)$, 
  we obtain $c_T \leq c < \infty$.
  
  Next, we prove $c_T >0$ by contradiction. Assume that $c_T = 0$. This yields
  $T(f) = 0$ for $H$-a.e.\ $f \in \HH$ which -- by the max-linearity of $Z$ --
  implies $T(Z) = 0$ a.s.\ in contradiction to $\inf_{y \in K} Z(y) > 0$ a.s.\
  and \eqref{eq:F-bounds}. Thus, we conclude that $c_T \in (0,\infty)$.  
  
  Now, let $g \in D$. Using the max-stability of $T$, we have
  \begin{align}
   \Qproxi_{g} -1 
  ={}& \sum_{m=1}^\infty \PP\left( \esssup_{f \in \HH} \sup_{y \in K} 
  \frac{f(y)}{g(f)} \bigg/ \sum_{j=1}^{m+1} E_j > \max_{1 \leq k \leq m} \frac 1 {\sum_{j=1}^k E_j}
  \frac{T(f_k)}{g(f_k)}  \right) \nonumber\\
  ={}& \sum_{m=1}^\infty \PP\left( \frac{\sum_{j=1}^k E_j}{\sum_{j=1}^{m+1} 
  E_j} > \frac{T(f_k)}{g(f_k)} \bigg/ \esssup_{f \in \HH} \sup_{y \in K} \frac{f(y)}{g(f)},
  \ 1 \leq k \leq m \right). \label{eq:ej}
  \end{align}
  Note that 
  $$\frac{T(f_k)}{g(f_k)} \Big/ \esssup_{f \in H} \sup_{y \in K} \frac{f(y)}{g(f)}
  \in [0,1].$$ As the joint distribution of 
  $(\sum_{j=1}^k E_j \big/ \sum_{j=1}^{m+1} E_j)_{k=1,\ldots,m}$ equals the
  joint distribution of the order statistics of $m$ independent random 
  variables which are uniformly distributed on $[0,1]$, we obtain  
  \begin{align} 
  & \Qproxi_{g} {}={} 1 + \sum_{m=1}^\infty 
  \bigg[1 - \EE \left(\frac{T(f_1)}{g(f_1)} \right)
  \Big/ \esssup_{f \in \HH} \sup_{y\in K} \frac{f(y)}{g(f)}\bigg]^m \label{eq:\Qproxi-calc}
  \displaybreak[0]\\
  ={} \esssup_{f \in \HH} & \sup_{y\in K} \frac{f(y)}{g(f)} \Big/  
  \EE \left(\frac{T(f_1)}{g(f_1)} \right)
  ={} \esssup_{f\in\HH} \frac{\sup_{y \in K} f(y)}{g(f)} \Big/
  \int_{\HH} T(f_1) \, H(\dd f_1). \nonumber
  \end{align}
  This finishes the proof since $c_T \in (0,\infty)$. 
\end{proof}

\begin{remark}
 From the proof of Proposition \ref{prop:simplification} we get that the random
 variable $M = \min\{m \in \NN: \eqref{eq:cond-maxlin} \text{ is satisfied}\}$
 follows a geometric distribution with parameter
 $\int_{\HH} T(h) H(\dd h) \big/ \esssup_{f\in\HH} \sup_{y\in K} (f(y)/g(f))$.
 Therefore, minimizing $\esssup_{f \in \HH} \sup_{y \in K} (f(y)/g(f))$
 will not only minimize the expectation of $M$, but also other characteristics
 such as $\PP(M > m_0)$ for $m_0 \in \NN$. However, this property may not hold
 for the stochastic number $m$ in the actual stopping rule 
 \eqref{eq:stopping-trafo}.
\end{remark}

We carry on to find the density $g$ that minimizes 
$\esssup_{f \in \HH} \frac{\sup_{y \in K} f(y)}{g(f)}$.
Instead of considering $f \mapsto \sup_{y \in K} f(y)$ in the numerator, the
following theorem deals with a broader class of functionals.

\begin{theorem} \label{thm:argmin-element}
Assume we are in the framework of Proposition \ref{prop:1}.
Let $L: \HH \to (0,\infty)$ be measurable and
$c_L := \int_{\HH} L(f) H(\dd f) < \infty$. 
Then, 
\begin{equation} \label{eq:optimal-g}
  g^*(f) = c_L^{-1} L(f)
\end{equation}
is an element of 
$$ \Qproxi_{L} = \arg\min_{g} \esssup_{f \in \HH} \frac{L(f)}{g(f)}.$$    
Furthermore, for every $g \in \Qproxi_{L}$, Equation \eqref{eq:optimal-g} holds
for $H$-a.e.\ $f \in \HH$.
\end{theorem}

\begin{proof}
 First, by contradiction, we show that the inequality 
 \begin{equation} \label{eq:opt-bound}
 \esssup_{f \in H} \frac{L(f)}{g(f)} \geq c_L 
 \end{equation}
 holds for all $g$. So, assume that \eqref{eq:opt-bound} does not hold for some
 $g$ considered in Proposition \ref{prop:1}. Then some $\varepsilon > 0$ and
 some density $g$ with $\int g(f) \, H(\dd f)= 1$ exist such that, for all
 $f \in \HH$, we have $ L(f) / g(f) \le c_L - \varepsilon.$
 Hence,
 $$c_L = \int_\HH L(f) \, H(\dd f) \le \left(c_L - \varepsilon \right) 
       \int_\HH g(f) H(\dd f) < c_L $$
 which is a contradiction. Hence, \eqref{eq:opt-bound} is proved. Note that the
 choice $g(f) = c_L^{-1} L(f)$ implies equality in \eqref{eq:opt-bound}.
 The first assertion follows.
  
 For the proof of the second assertion, assume that there is some 
 $g \in \Qproxi_{L}$ such that \eqref{eq:optimal-g} does not hold for $H$-a.e.\
 $f \in \HH$. Then, as
 $$\int_\HH g(f) \,H(\dd f) = 1 = \int_{\HH} c_L^{-1} L(f) H(\dd f),$$
 we get that there is some set $A \subset \HH$ with $H(A) > 0$ such that, for 
 all $f \in A$, $g(f) < c^{-1} L(f)$, but $g(f) > 0$ by 
 \eqref{eq:g-regularity}. This yields $gH(A) > 0$ and, hence, 
 $ \esssup_{f \in \HH} L(f)/g(f) > c_L$, 
 which is a contradiction to $g \in \Qproxi_{L}$.
\end{proof}

The results stated above enable us to give a necessary and sufficient
condition for the solvability of the optimization problem \eqref{eq:simple}
and to describe its solution. Here, we call an optimization problem
$$\textstyle \arg\min_{x \in A} h(x), \quad h: A \to \RR \cup \{\infty\},$$
\emph{solvable} if $\inf_{x \in A} h(x) \in (-\infty, \infty)$ and if there
exists some $x_0 \in A$ such that $h(x_0) = \inf_{x \in A} h(x)$.

\begin{corollary} \label{coro:simple-special}
 Let $f \mapsto \sup_{y \in K} f(y)$ be measurable. Then, the
 optimization problem \eqref{eq:simple} is solvable if and only if
  $$c = \int_\HH \sup_{y \in K} f(y) H(\dd f) < \infty.$$ In this case,
 $$g^* \in \Qproxi = \arg\min_{g} \Qspec1_{g} = \arg\min_{g} \Qspec2_{g}(y_0)$$
 with $g^*$ as defined in \eqref{eq:g-star}, that is, the normalized spectral
 representation is optimal w.r.t.\ \eqref{eq:simple}. The solution is
 unique $H$-a.s.
\end{corollary}
\begin{proof}
 If $c=\infty$, Equation \eqref{eq:c-bound} and Proposition 
 \ref{prop:simplification} yield that \eqref{eq:simple} is not solvable.
 For $c<\infty$, the assertion follows directly from Proposition 
 \ref{prop:simplification} and Theorem \ref{thm:argmin-element} with
 $L(f) = \sup_{y \in K} f(y)$ and
 $T(h) = \sup_{y\in K} h(y)$ and $T(h) = h(y_0)$, respectively.
\end{proof}

\begin{remark}
 Note that the solution of the optimization problem \eqref{eq:simple}
 is unique in two different aspects. First, any solution $g \in \Qproxi$
 satisfies $g=g^*$ $H$-a.s. Second, due to the uniqueness of the normalized
 spectral representation (Proposition \ref{prop:uniqueness}), the 
 finite-dimensional distributions of the spectral functions $\{f/g(f)\}$
 do not depend on the initial choice of the spectral functions, i.e.\ on the
 choice of the space $\HH$ and the measure $H$.
\end{remark}

\begin{remark} \label{rem:inf-number}
 Corollary \ref{coro:simple-special} yields that $\Qspec1_{g} = \infty$ if $Z$
 does not allow for a normalized spectral representation (or, equivalently,
 $c=\infty$). As the definitions imply that $\Qspec1_{g} \leq \Qorig_g$ for any
 $g$, there is no spectral representation such that the expected number of 
 points which are able to contribute to $Z$ is finite in this case.
\end{remark}

\section{Evaluating the modified optimization problem} \label{sec:refine}

In this section, we discuss how close the relation is between the modified
optimization problem and the original problem. Observing that
$\Qspec1_g \leq \Qspec2_g(y_0) \leq \Qorig_g$ for all $g$ (see the proof of
Proposition \ref{prop:bound-1}), we see that the modified optimization problem
is in fact minimizing a lower bound function of the mapping 
$g \mapsto \Qorig_g$. We will improve the lower bound and show that the
normalized spectral representation also minimizes the improved lower bound
function. In addition, we give a formula for calculating $\Qorig_g$. In
particular, this formula allows for the calculation of $\Qorig_{g^*}$, that is,
the expected number of points which are able to contribute to $Z$ in the
normalized spectral representation regarding the actual stopping rule
\eqref{eq:stopping-trafo}. In the following proposition, we especially look at
$\Qspec1_{g^*}$, $\Qspec2_{g^*}(y_0)$ and $\Qorig_{g^*}$ to get bounds for the
real optimal solution.

\begin{proposition} \label{prop:bound-1}

 \begin{itemize}
  \item[1. ] The function $y_0 \mapsto \Qspec2(y_0)$ is constant on $K$.
  \item[2. ] $\Qspec1_g \leq \Qorig_g$ and $\Qspec1_g \leq \Qspec2_g(y_0)$, 
             $y_0 \in K$, for all $g$.
  \item[3. ] $\Qspec1_{g^*} = 1$.
  \item[4. ] Assuming further that there is a countable set 
     $K_0 \subset K$ such that
     $\sup_{y \in K} f(y) = \sup_{y \in K_0} f(y)$ for $H$-a.e.\ $f \in \HH$,
     we obtain that $$\Qspec2_g(y_0) \leq \Qorig_g$$
     for all $g$ and all $y_0 \in K$.
     In particular, for any solution of the original optimization problem,
     $\tilde g \in \Qorig$, we get the bounds
     $$ 1 = \Qspec1_{g^*} \leq \Qspec2_{g^*}(y_0) \leq \Qorig_{\tilde g}
        \leq \Qorig_{g^*}.$$
 \end{itemize}
\end{proposition}

\begin{proof}
 First, we note that 
 \begin{align*}
  \esssup_{f\in\HH} \frac{\sup_{y \in K} f(y)}{g(f) \sup_{\tilde y \in K} Z^{(m)}(\tilde y)}
    \leq{} & \esssup_{f\in\HH}  \frac{\sup_{y \in K} f(y)}{g(f) Z({y})} \displaybreak[0]\\
 \text{and} \  \esssup_{f\in\HH}  \frac{\sup_{y \in K} f(y)}{g(f) \sup_{\tilde y \in K} Z^{(m)}(\tilde y)}
    \leq{} & \esssup_{f\in\HH}  \frac{\sup_{y \in K} f(y)}{g(f) Z^{(m)}(y_0)}
 \end{align*}
 for any $g$ and any $y_0 \in K$. 
 By \eqref{eq:std-frechet}, we have $ \int_\HH  f(y) H(\dd f) = 1$ for all 
 $y \in K$, and thus, by \eqref{eq:\Qproxi-calc}, 
 $y_0 \mapsto \Qspec2_{g}(y_0)$ is constant on $K$ for any $g$. 
 This proves the first two parts of the proposition.
 
 To see the third assertion, we note that we have
 \begin{equation} \label{eq:sup-y}
  \sup_{y \in K} \frac{f(y)}{g^*(f)}=c \quad \text{for } g^*H\text{-a.e.\ } f \in \HH
 \end{equation}
  and, thus, with \eqref{eq:z-m}, we obtain
 \begin{equation} \label{eq:sup-proc}
  \sup_{y \in K} Z^{(m)}(y) 
  = \sup_{y \in K} \max_{i \in \NN} \frac 1 {\sum_{j=1}^i E_j} \frac{f_i(y)}{g^*(f_i)} = c E_1^{-1} \quad \text{for all } m \in \NN.
 \end{equation}
 Equations \eqref{eq:sup-y} and \eqref{eq:sup-proc} yield
 \begin{align*}
  \Qproxi_{g^*} ={} & \EE \min\bigg\{m\in \NN: \esssup_{f \in \HH} \sup_{y\in K} \frac{f(y)}{g^*(f)}
                              \leq \sup_{y\in K} Z^{(m)}(y) \sum_{j=1}^{m+1} E_j \bigg\}\\
      ={} & \textstyle \EE \min\left\{m\in \NN: c \leq c\cdot\ \sum_{j=1}^{m+1} E_j / E_1 \right\} = 1.
 \end{align*}
 
 For the proof of the fourth part of the proposition, we assume that there 
 exists some countable set $K_0 \subset K$ such that
 $\sup_{y \in K} f(y) = \sup_{y \in K_0} f(y)$ for $H$-a.e.\ $f \in \HH$.
 Hence, we have that 
 \begin{equation} \label{eq:countable}
  \esssup_{f \in \HH} \sup_{y \in K} \frac{f(y)}{g(f)}
 = \sup_{y \in K_0} \esssup_{f \in \HH} \frac{f(y)}{g(f)}.
 \end{equation}
 We first consider the case that
 $\esssup_{f \in \HH} \sup_{y \in K} \frac{f(y)}{g(f)} = \infty$.
 Then, either $c = \int_{\HH} \sup_{y \in K} f(y) H(\dd f) = \infty$, which 
 yields $\infty = \Qspec1_{g^*} \leq \Qorig_g$ (cf.\ Remark 
 \ref{rem:inf-number}), or $c < \infty$. By
 Proposition \ref{prop:c-finite}, the latter implies that
 $\sup_{y \in K} Z(y) < \infty$ with probability one. Thus, by the definition
 of $\Qspec1_g$ in \eqref{eq:simple-1}, $\infty = \Qspec1_g \leq \Qorig_g$.
 The only case that remains is that 
 $\esssup_{f \in \HH} \sup_{y \in K} \frac{f(y)}{g(f)} < \infty$. Then, 
 by \eqref{eq:countable}, for every $\varepsilon > 0$, there exists some
 $y_0(\varepsilon)$ such that
 \begin{align}
  \frac{1}{1+\varepsilon} 
  \esssup_{f \in \HH} \sup_{y \in K} \frac{f(y)}{g(f) Z^{(m)}(y_0(\varepsilon))}
 \leq{}& \esssup_{f \in \HH} \frac{f(y_0(\varepsilon))}{g(f) Z^{(m)}(y_0(\varepsilon))} \nonumber\\
 \leq{}& \esssup_{f \in \HH} \sup_{y \in K} \frac{f(y)}{g(f) Z^{(m)}(y)} \label{eq:eps-assess}
 \end{align}
 with probability one. Further, analogously to the proof of Proposition
 \ref{prop:simplification},
 \begin{align*}
 & \EE \min\bigg\{ m\in \NN: \frac 1 {1+\varepsilon} \esssup_{f\in\HH} 
   \frac{\sup_{y \in K} f(y)}{g(f) Z^{(m)}(y_0)} \le \sum_{j=1}^{m+1} E_j \bigg\} \displaybreak[0]\\
={} & 1 + \sum_{m=1}^\infty 
  \bigg[1 - \EE \bigg(1 \wedge \bigg(\frac{(1+\varepsilon)\cdot f_1(y_0)}{g(f_1)} \Big/ 
                      \esssup_{f \in \HH} \sup_{y\in K} \frac{f(y)}{g(f)}\bigg)\bigg)\bigg]^m \nonumber \displaybreak[0]
  \\={}& \bigg[ \EE \bigg( 1 \wedge \bigg(\frac{(1+\varepsilon)\cdot f_1(y_0)}{g(f_1)} \Big/ 
                           \esssup_{f \in \HH} \sup_{y\in K} \frac{f(y)}{g(f)}\bigg)\bigg)\bigg]^{-1}
  \geq{}  \frac 1 {1+\varepsilon} \Qspec2_g(y_0)  
 \end{align*}
 for all $y_0 \in K$, $\varepsilon > 0$ and all $g$, and thus,
 with \eqref{eq:eps-assess},
 $$\textstyle (1+\varepsilon)^{-1} \inf_{y_0 \in K} \Qspec2_{g}(y_0) \leq \Qorig_g$$
 holds for all $\varepsilon > 0$. Hence, the fourth part of the proposition follows.
\end{proof}

As Proposition \ref{prop:bound-1} shows, the approximation of the optimal 
function value in original problem \eqref{eq:original} by \eqref{eq:simple-1}
and \eqref{eq:simple-2} might be quite vague. In particular, the minimum of
$\Qspec1_{g}$ always equals $1$, that is, some spectral functions which in fact
contribute to $Z$ are not considered. Replacing the processes $Z^{(m)}$ 
occurring in the  construction by the final process $Z^{(\infty)}$ given by
\eqref{eq:z-infty} allows us to take into account all those shape functions
which contribute to $Z^{(\infty)}$. To this end, we revisit the notion of 
\emph{$K$-extremal} and \emph{$K$-subextremal points} introduced by 
\cite{DEM13} and \cite{dombryeyiminko12}.
Henceforth, we suppose that the following assumption holds true which enables
us to consider this problem.
\begin{assumption} \label{as:sep-maxlin}
 Let $\HH$ satisfy the following conditions:
\begin{enumerate}
 \item[(i)] There exists some countable set $K_0 \subset K$ such that, for all
      $h_1, h_2 \in \HH$,
      $$ h_1 < h_2 \text{ on } K 
     \iff \exists \varepsilon > 0: h_1(y) < h_2(y) - \varepsilon 
      \text{ for all } y \in K_0.$$
 \item[(ii)] $\HH$ is a max-linear space, i.e.\ 
       $$t_1 h_1 \vee t_2 h_2 \in \HH$$ for all $h_1, h_2 \in \HH$, $t_1, t_2 \geq 0$.
       Further, $\1 \in \HH$.
 \item[(iii)] $\HH$ is endowed with the $\sigma$-algebra $\H$ of cylinder sets, 
       $\H = \sigma(\{h \in \HH: \ h(y) \in B\}, \ y \in K, \ B \in \B \cap [0,\infty))$.
\end{enumerate}
\end{assumption}

\begin{remark}
\begin{enumerate}
 \item  Assumption \ref{as:sep-maxlin} implies that the additional 
 assumption in the fourth part of Proposition \ref{prop:bound-1} holds, i.e.\
 there exists a countable set $K_0 \subset K$ such that
 $\sup_{y \in K} f(y) = \sup_{y \in K_0} f(y)$ all $f \in \HH$.
 \item As $\H$ is the $\sigma$-algebra of cylinder sets, Assumption
 \ref{as:sep-maxlin} ensures that the process  $F^*$ in the normalized spectral
 representation \eqref{eq:spec-repr} is  uniquely determined (cf.\ Proposition
 \ref{prop:uniqueness}). 
\end{enumerate}
\end{remark}

\begin{definition} \label{def:plus-minus}
 Let $\Phi$ be some Poisson point process on $(0,\infty) \times \HH$
 with intensity measure $u^{-2} \DD u \times \nu(\dd h)$
 for some locally finite measure $\nu$ on $\HH$.
 We call $(t^*,h^*) \in \Phi$ a $K$-extremal point and write $(t^*,h^*) \in
 \Phi_K^{+}$ if and only if
 $$ t^* h^*(y) = \max_{(t,h) \in \Phi} t h(y) \quad \text{ for some } y \in K.$$
 
 Otherwise, i.e.\ if $t^* h^*(y) < \max_{(t,h) \in \Phi} t h(y)$ for all
 $y \in K$, $(t^*,h^*) \in \Phi$ is called a $K$-subextremal point and we write
 $(t^*,h^*) \in \Phi_K^{-}$.
\end{definition}

We generalize a result given in \cite{DEM13} and show that the random
sets $\Phi_K^+$ and $\Phi_K^-$ are point processes on $(0,\infty) \times \HH$,
i.e.\ $\Phi_K^+(C)$ and  $\Phi_K^-(C)$ are random variables for any bounded
set $C \in \B \times \H$. In contrast to \cite{DEM13}, we are interested in
tuples $(t,h)$ instead of the product $th$ and we do not restrict to continuous
functions. Nevertheless, the proof runs analogously.

\begin{proposition} \label{prop:pp}
 $\Phi_K^+$ and $\Phi_K^-$ are point processes on $(0,\infty) \times \HH$. 
\end{proposition}
\begin{proof}
 First, we note that, for $\HH_0 = (0,\infty) \times \HH$, the mapping 
 $\phi: \HH_0 \to \HH, \ (t,h) \mapsto t h(\cdot)$ is measurable.
 Therefore, events of the type 
 $\{\omega \in \Omega: \, \Phi(\{ (t,h) \in \HH_0: t h \in C\}) = k\}$
 are measurable for any $C \in \H$ and $k \in \NN_0$.
 
 Now, let $u_0 > 0$, $C \in \H$ with $\nu(C) < \infty$ and $k \in \NN_0$.
 Furthermore, let $K=\{x_1, x_2, \ldots\}$ be as in Assumption 
 \ref{as:sep-maxlin}. Then, the event
 \begin{align*}
  & \{ \omega \in \Omega: \Phi_K^-((u_0,\infty) \times C) \geq k \}\\
 ={} & \bigcup_{\varepsilon \in \QQ_+} \bigcap_{n \in \NN} 
       \bigcup_{\mathbf{q} \in \QQ_+^n}  \Big\{ \omega \in \Omega: \,
       \Phi\Big(\big\{(t,f) \in \HH_0: \, tf(x_i) > q_i\big\} \Big)
       \geq 1, \ 1 \leq i \leq n, \\
 &  \hspace{0.7cm} 
    \Phi\Big(((u_0,\infty) \times C) \cap \big\{ (t,f) \in \HH_0: 
              \, tf(x_j) < q_j - \varepsilon, \, 1 \leq j \leq
              n\big\}\Big) \geq k \Big\} 
 \end{align*}
 is measurable. 
 Thus, $\Phi_K^-$ and $\Phi_K^+ = \Phi \setminus \Phi_K^-$ are point processes.
\end{proof}

For applying the theory of extremal and subextremal points to the construction
of the process $Z^{(\infty)}$, we define the Poisson point process
$$\Phi = \left\{\left(t,f/g(f)\right): \ (t,f) \in \tilde \Pi\right\}.$$
Then, similarly to the proof of Lemma 3.2 in \cite{dombryeyiminko12}, 
the following result on $\tilde \Pi_K^-$ can be shown.
\begin{lemma} \label{lem:piminus}
Conditional on $Z^{(\infty)}$, the point process $\Phi_K^{-}$
is a Poisson point process on $(0,\infty) \times \HH$ with intensity measure
$$ \tilde \Lambda^{-}(\dd t, \dd h) = 
   t^{-2} \times \int_\HH \int_{T} \mathbf{1}_{f(\cdot) \in \DD h} \cdot
   \mathbf{1}_{t f(\cdot)/g(f) < Z^{(\infty)}(\cdot)}\, g(f) H(\dd f) \, \DD t.$$
\end{lemma}

In the following, we will mainly focus on the first component of the
point process $\Phi$, i.e.\ we consider the point processes
\begin{align*}
 \Pi_{0,K}^{+} ={} & \{ t: \ (t, f/g(f)) \in \Phi_K^{+}\}\\
 \text{and} \quad \Pi_{0,K}^{-} ={} & \{ t: \ (t, f/g(f)) \in \Phi_K^{-}\},
\end{align*}
respectively. Obviously, any $t^* \in \Pi_{0,K}^{+}$, is taken into account by the 
definition of $\Qorig_g$ in \eqref{eq:original}. Thus, we can rewrite
\eqref{eq:original} as
\begin{equation} \label{eq:original-ref}
\Qorig_{g} = \EE |\Pi_{0,K}^{+}| + \EE\bigg( \bigg| \bigg\{t \in \tilde \Pi_{0,K}^{-}: 
                   \ \esssup_{f \in \HH} \sup_{y \in K} 
                     \frac{f(y)}{g(f) Z^{(\infty)}(y)} > \frac 1 t \bigg\} \bigg| \bigg).
\end{equation}
Including all the points of $\Pi_{0,K}^{+}$, i.e.\ all the spectral functions that
finally contribute to $\tilde Z^{(\infty)}$, and replacing $Z^{(m)}$ by 
$Z^{(\infty)}$, we analogously obtain refined versions of \eqref{eq:simple},
\eqref{eq:simple-1} and \eqref{eq:simple-2} as
\begin{align}
\tildeQproxi_{g} ={} & \EE |\Pi_{0,K}^{+}| + \EE\bigg( \bigg| \bigg\{t \in \Pi_{0,K}^{-}: \,
                   \esssup_{f \in \HH} \frac{\sup_{y \in K} f(y)}{g(f) T(Z^{(\infty)})} > \frac 1 t
                   \bigg\} \bigg| \bigg) \label{eq:refined},\displaybreak[0]\\
\tildeQspec1_{g}
={} & \EE |\Pi_{0,K}^{+}| + \EE\bigg( \bigg| \bigg\{t \in \Pi_{0,K}^{-}: \,
               \esssup_{f \in \HH} \frac{\sup_{y \in K} f(y)}{g(f) \sup_{\tilde y \in K} Z^{(\infty)}(\tilde y)} > \frac 1 t
           \bigg\} \bigg| \bigg), \label{eq:refined-1}\displaybreak[0]\\ 
\tildeQspec2_{g}(y_0&) ={} 
 \EE |\Pi_{0,K}^{+}| + \EE\bigg( \bigg| \bigg\{t \in \Pi_{0,K}^{-}: \, \esssup_{f \in \HH} 
                     \frac{\sup_{y \in K} f(y)}{g(f) Z^{(\infty)}(y_0)} > \frac 1 t\bigg\} \bigg| \bigg). \label{eq:refined-2}
\end{align}

By definition, obviously $\tildeQspec1_{g} \geq \Qspec1_{g}$ and 
$\tildeQspec2_{g} \geq \Qspec2_{g}$ for all $g$. Hence, $\tildeQspec1_g$ and
$\tildeQspec2_g$ are improved lower bounds of $\Qorig_g$ (see Proposition
\ref{prop:bound-2} below). The optimization of these bounds is facilitated by
the following result on $\Pi_{0,K}^+$.

\begin{lemma} \label{lem:piplus}
 We have $$\EE |\Pi_{0,K}^{+}| 
 = \EE_{ Z}\left(\int_{\HH} \sup_{y \in K} \frac{f(y)}{Z(y)} \, H(\dd f) \right)$$
 which does not depend on the choice of $g$.
\end{lemma}
\begin{proof}
Let $B  = [t_0,\infty)$ with $t_0 >0$.
Then, we have
\begin{align*}
   \EE \left|\Pi_{0,K}^{+} \cap B \right|
={}& \EE\left| \tilde \Pi \cap (B \times \HH) \right| - \EE \left|\Pi_{0,K}^{-} \cap B\right|.
\end{align*}
 Conditioning on $Z^{(\infty)}$, Lemma \ref{lem:piminus} yields 
\begin{align*}
     \EE \left|\Pi_{0,K}^{+} \cap B\right|
={} & \int_\HH \int_0^\infty t^{-2} \mathbf{1}_{t \geq t_0} \DD t \, g(f) H(\dd f) \displaybreak[0]\\
     - \EE_{Z^{(\infty)}} & \left(\int_\HH \int_0^\infty t^{-2} \mathbf{1}_{t \geq t_0} 
        \mathbf{1}_{\frac 1 t > \sup_{y \in K} \frac{f(y)}{g(f) Z^{(\infty)}(y)}} \DD t \, g(f) H(\dd f)\right)\displaybreak[0]\\
={} & \EE_{Z} \left(\int_\HH \int_0^\infty t^{-2} \mathbf{1}_{t \geq t_0} 
        \mathbf{1}_{\frac 1 t \leq \sup_{y \in K} \frac{f(y)}{g(f) Z(y)}} \DD t \, g(f) H(\dd f)\right).    
\end{align*}
Considering a monotone sequence $t_{0,n} \searrow 0$ as $n \to \infty$, the
monotone convergence theorem yields
\begin{align*}
 \EE|\Pi_{0,K}^+| ={} & \EE_{Z} \left( \int_\HH \int_0^\infty t^{-2} \mathbf{1}_{t > 1/\sup_{y \in K} \frac{f(y)}{g(f) Z(y)}}
   \DD t \, g(f) H(\dd f) \right)\\
 ={} & \EE_{Z}\left(\int_{\HH} \sup_{y \in K} \frac{f(y)}{Z(y)} \, H(\dd f) \right),
\end{align*}
  which completes the proof.
\end{proof}

The results stated in Lemma \ref{lem:piminus} and \ref{lem:piplus}
facilitate the calculation of $\Qorig_{g}$ and allow us to relate the
minimizer of \eqref{eq:refined} to the solution of our previously
modified optimization problem, $g^* \in \Qproxi$.

\begin{proposition} \label{prop:refinement}
 \begin{enumerate}
 \item For any $g$, we have
   \begin{equation} \label{eq:Qorig-form}
    \Qorig_{g} = \EE_{Z} \bigg(\esssup_{f \in \HH} \sup_{y \in K} \frac{f(y)}{g(f) Z(y)} \bigg).
   \end{equation}
 \item Assume that $f \mapsto \sup_{y \in K} f(y)$ is measurable. Then, with $\tildeQproxi_{g}$ as in 
   \eqref{eq:refined}, for any max-linear function $T$, it holds
   $$\arg\min_{g} \tildeQproxi_{g} 
    \supset \arg\min_{g} \esssup_{f \in \HH} \frac{\sup_{y \in K} f(y)}{g(f)} = \Qproxi,$$
   where $\Qproxi = \arg\min_{g} \Qproxi_{g}$.
\end{enumerate}
\end{proposition}
\begin{proof}
Let $Z^{(\infty)}$ be given by \eqref{eq:z-infty}.
By Lemma \ref{lem:piplus}, we obtain
 \begin{align*}
 \Qorig_{g} ={} & \EE_{Z} \ \int_\HH \sup_{y \in K} \frac{f(y)}{Z(y)} H(\dd f)\displaybreak[0]\\
     & + \EE\bigg( \bigg| \bigg\{t \in  \Pi_{0,K}^{-}: \ \esssup_{f \in \HH} \sup_{y \in K} 
                     \frac{f(y)}{g(f) Z^{(\infty)}(y)} > t^{-1} \bigg\} \bigg| \bigg).  
 \end{align*}
 Conditioning on $Z^{(\infty)}$, Lemma \ref{lem:piminus} yields                    
 \begin{align}
  & \EE\bigg( \bigg| \bigg\{t \in \Pi_{0,K}^{-}: 
                   \ \esssup_{f \in \HH} \sup_{y \in K} \frac{f(y)}{g(f) Z^{(\infty)}(y)} > t^{-1} \bigg\} \bigg| \bigg) 
    \nonumber \displaybreak[0]\\
  ={} & \EE_{Z} \bigg(\int_{\HH} \int_0^\infty t^{-2} \mathbf{1}_{t>1/{\esssup\limits_{h \in \HH}} \sup\limits_{y \in K} 
     \frac{h(y)}{g(h) Z(y)}}
     \mathbf{1}_{t<1/\sup\limits_{y \in K} \frac{f(y)}{g(f) Z(y)}}  \DD t \, g(f) H(\dd f)\bigg) 
     \nonumber \displaybreak[0]\\
   ={} & \EE_{Z} \left( \int_{\HH}  \bigg\{\esssup_{h \in \HH} \sup_{y \in K} \frac{h(y) g(f)}{g(h)Z(y)}
             - \sup_{y \in K} \frac{f(y)}{Z(y)} \bigg\}_{+} \, H(\dd f) \right)
        \nonumber \displaybreak[0]\\
   ={} & \EE_{Z} \bigg[\esssup_{h \in \HH} \sup_{y \in K} \frac{h(y)}{g(h) Z(y)} \bigg]
      - \EE_{Z} \int_\HH \sup_{y \in K} \frac{f(y)}{Z(y)} \,H(\dd f).
   \nonumber
 \end{align}
 In the last step we used the fact that 
 $$ \esssup_{h \in \HH} \sup_{y \in K} \frac{h(y)g(f)}{g(h) Z(y)} - \sup_{y \in K} \frac{f(y)}{Z(y)} \geq 0$$
 for $H$-a.e.\ $f \in \HH$. The first assertion follows.
 
 Analogously to the first part, we get that
 \begin{align*}
 \tildeQproxi_{g} ={} & \EE_{Z} \int_\HH \sup_{y \in K} \frac{f(y)}{\tilde Z(y)} H(\dd f)\\
    & + \EE\bigg( \bigg| \bigg\{t \in \Pi_{0,K}^{-}: 
                   \ \esssup_{f \in \HH} \frac{\sup_{y \in K} f(y)}{g(f) T(Z^{(\infty)})} > t^{-1} \bigg\} \bigg| \bigg)
 \end{align*}
 and 
 \begin{align}
  & \EE\bigg( \bigg| \bigg\{t \in \Pi_{0,K}^{-}: 
                   \, \esssup_{f \in \HH} \frac{\sup_{y \in K} f(y)}{g(f) T(Z^{(\infty)})} > t^{-1} \bigg\} \bigg| \bigg) \nonumber \\
  ={} & \EE_{Z} \bigg(\int_{\HH} \int_0^\infty t^{-2} 
     \mathbf{1}_{t>1/\esssup\limits_{h \in \HH} \sup\limits_{y \in K} \frac{h(y)}{g(h) T(Z)}}
     \mathbf{1}_{t<1/\sup\limits_{y \in K} \frac{f(y)}{g(f) \tilde Z(y)}} \DD t \, g(f)  H(\dd f)\bigg) \nonumber \displaybreak[0]\\
  ={} & \EE_{Z} \left( \int_{\HH} \bigg\{\esssup_{h \in \HH} \frac{\sup_{y \in K} h(y)}{g(h) T(Z)}
         g(f)  - \sup_{y \in K} \frac{f(y)}{Z(y)} \bigg\}_+ H(\dd f) \right) \nonumber \displaybreak[0]\\
  \geq{} & \EE_{Z} \left( \int_{\HH} \bigg\{ \frac{\sup_{y \in K} f(y)}{ T(Z)}
         - \sup_{y \in K} \frac{f(y)}{Z(y)} \bigg\}_+ H(\dd f) \right) \label{eq:refined-calc}.
 \end{align}
 Now, let $g \in \Qproxi=\arg\min_{g} \esssup_{f \in \HH} (\sup_{y \in K} f(y) / g(f))$.
 Then, by Theorem \ref{thm:argmin-element}, we have that 
 $g(f) = c^{-1} \sup_{y \in K} f(y)$ for $H$-almost all $f \in \HH$. Thus,
 we get equality in Equation \eqref{eq:refined-calc} and hence 
 $\Qproxi \subset \arg\min_{g} \tildeQproxi_{g}$.
\end{proof}

Proposition \ref{prop:refinement} leads to two implications in application.
Firstly, it allows for the numerical calculation of 
$\Qorig_{g^*}$ for any given max-stable process $Z$. With
$\frac{f(y)}{g^*(f)}  \leq c$, we get the assessment
\begin{equation} \label{eq:q-assess}
 \Qorig_{g^*} \leq{} c \cdot \EE\Big[\Big(\inf_{y \in K} Z(y)\Big)^{-1}\Big].
\end{equation}
Under the assumption that condition \eqref{eq:infz-finite} holds (see Remark 
\ref{rem:infz-finite} for a sufficient condition), this yields that
$\Qorig_{g^*} < \infty$ if $c<\infty$. In other words, the expectation of the
stochastic number $m$ from  \eqref{eq:stopping-trafo} is then finite for the
normalized representation. 

We further evaluate when $\Qorig_{g^*}$ reaches its upper bound as in 
\eqref{eq:q-assess}. Note that equality in \eqref{eq:q-assess} holds if and 
only if 
\begin{equation} \label{eq:cond-sharp}
\esssup_{f \in \HH} \sup_{y \in K} \frac{f(y)}{g^*(f) Z(y)} = \frac{c}{\inf_{\tilde y \in K} Z(\tilde y)} \quad a.s.
\end{equation}
Furthermore, by Assumption \ref{as:sep-maxlin}, $\sup_{y \in K}$ in 
\eqref{eq:cond-sharp} may be replaced by $\sup_{y_0 \in K}$ for some countable
set $K_0$. The fact that $K_0$ is countable allows for interchanging 
$\esssup_{f \in \HH}$ and $\sup_{y \in K_0}$, i.e.\ Equation \eqref{eq:cond-sharp}
is equivalent to
\begin{equation} \label{eq:cond-sharp-2}
\sup_{y \in K_0} \esssup_{f \in \HH} \frac{f(y)}{g^*(f) Z(y)} = \frac{c}{\inf_{\tilde y \in K} Z(\tilde y)} \quad a.s.
\end{equation}
Thus, condition \eqref{eq:cond-sharp} holds if and only if, with probability one, there exists a sequence $(y_n)_{n \in \NN}$ in $K_0$ such that
\begin{equation} \label{eq:same-limit}
\lim_{n \to \infty} \esssup_{f \in \HH} \frac{f(y_n)}{\sup_{\tilde y \in K} f(\tilde y)}  = \lim_{n \to \infty} \frac{Z(y_n)}{\inf_{\tilde y \in K} Z(\tilde y)}.
\end{equation}
Note that the left-hand side of \eqref{eq:same-limit} is bounded from above by $1$, while the right-side is bounded
from below by $1$. Condition \eqref{eq:same-limit} can be reformulated in the following way:
For every $\varepsilon > 0$ and almost every sample path of $Z$, there exists some $y \in K$ with
\begin{align} 
 & Z(y) < (1+\varepsilon) \inf_{\tilde y \in K} Z(\tilde y) \label{eq:cond-sharp-new1}\\
 \text{and} \quad & H\Big(\Big\{f \in \HH: \, f(y) > (1-\varepsilon) \sup_{\tilde y \in K} f(\tilde y)\Big\}\Big) > 0.\label{eq:cond-sharp-new2}
\end{align}

Analogously to Equation \eqref{eq:q-assess}, where $\EE m$ is bounded from 
above, the number $m$ can be bounded from above a.s.\ by
\begin{equation} \label{eq:m-assess}
 \min \left\{\tilde m \in \NN: \ c  / \inf_{y \in K} Z(y) \leq \textstyle \sum_{j=1}^{\tilde m +1} E_j\right\}
\end{equation}
and, again, $m$ equals \eqref{eq:m-assess} a.s.\ if and only if \eqref{eq:cond-sharp-new1} and
\eqref{eq:cond-sharp-new2} hold.

\begin{remark}
 If $Z$ is represented by a stochastic process, i.e. $Z$ is defined as in
 \eqref{eq:incremental} \citep[cf.][for example]{penrose92}, $H$ is a 
 probability measure, namely the law of $W$. In this case, condition \eqref{eq:cond-sharp-new2}
 is equivalent to
 \begin{equation} \label{eq:cond-sharp-incr}
  \PP\left( W(y) > (1-\varepsilon)\textstyle\sup_{\tilde y \in K} W(\tilde y)\right) > 0.
\end{equation}
\end{remark}
\medskip

Secondly, Proposition \ref{prop:refinement} implies that minimizing 
$\tildeQproxi_g$ can be achieved by any $g^* \in \Qproxi$.
We thus obtain the following corollary.

\begin{corollary} \label{coro:refine-special}
 Let the mapping $f \mapsto \sup_{y \in K} f(y)$ be measurable.
 Then, the optimization problem given in \eqref{eq:refined} is solvable
 if and only if the optimization problem \eqref{eq:simple} is solvable
 (cf.\ Corollary \ref{coro:simple-special}).
 In this case, we have
 $$ \Qproxi \subset \arg\min_{g} \tildeQspec1_{g} = \arg\min_{g} \tildeQspec2_{g}(y_0).$$
 In particular, the normalized spectral representation is optimal w.r.t.\ \eqref{eq:refined}.
\end{corollary}

Analogously to Proposition \ref{prop:bound-1}, the following result can be
shown.

\begin{proposition} \label{prop:bound-2}
 For any $\tilde g \in \Qorig$, we have
 $$\textstyle
   1 \leq \tildeQspec1_{g^*} 
     \leq \inf_{y_0 \in K} \tildeQspec2_{g^*}(y_0)
     \leq \Qorig_{\tilde g} \leq \Qorig_{g^*},$$
 where $g^* \in \Qproxi$ is given by \eqref{eq:g-star}.
\end{proposition}
\medskip

\begin{remark}
 In view of Proposition \ref{prop:refinement}, it appears promising to replace
 $g^*$  by an element of
 $$\Qproxi_{0} = \arg\min_{g} \bigg\{ \esssup_{f \in \HH} \EE_{ Z} 
 \bigg( \sup_{y \in K} \frac{f(y)}{g(f)  Z(y)}\bigg) \bigg\}$$
 to improve the partial minimization of $\Qorig$.
 If the functional
 $$ \textstyle L_{0}: \HH \to (0,\infty), \quad f \mapsto \EE_{ Z} \left( \sup_{y \in K}  Z(y)^{-1} f(y)\right)$$
 is measurable and if 
 $$ \textstyle c_{0} = \int_\HH \EE_{ Z} \left( \sup_{y \in K}  Z(y)^{-1} f(y)\right) \, H(\dd f) < \infty,$$
 then, by Theorem
 \ref{thm:argmin-element}, an element of $\Qproxi_{0}$ is given by
 \begin{equation} \label{eq:optimal-g-xi-altern}
  \textstyle g_{0}(f) = (c_{0}(f))^{-1} \EE_{ Z} \left( \sup_{y \in K}  Z(y)^{-1} f(y) \right).
 \end{equation}
 Conversely, for every $g \in \Qproxi_{0}$, \eqref{eq:optimal-g-xi-altern}
 holds for $H$-a.e.\ $f \in \HH$.
 
 Note that the calculation of $g_{0}$ is much more laborious than the 
 calculation of $g^*$, as the former one requires the computation of the
 expectation of  $\sup_{y \in K} ( Z(y)^{-1} f(y))$.
 However, computational experiments in case of the original Smith process 
 \citep{smith90} on finite intervals $[-b,b] \subset \RR$ indicate that 
 $\Qorig_{g_{0}}$ is not significantly smaller than $\Qorig_{g^*}$.
 Thus, the usage of $g^*$ seems to be preferable over the usage of $g_{0}$
 due to its accessibility.
\end{remark}

\section{Examples for the normalized spectral representation} \label{sec:examples}

In this section, we will investigate some specific cases for the process $Z$
from Proposition \ref{prop:1} and for the index set $K$. 
Under the general assumption that the mapping 
$f \mapsto \tilde f = \sup_{y \in K} f(y)$ is measurable, we consider the
normalized spectral representation $\tilde Z =_d Z$ in \eqref{eq:z-tilde}.
For some examples, we explicitly calculate $\tilde f$ and $c$ which are crucial
for the stopping rule and the expected number $\EE m$ of considered spectral
functions (cf.\ Equations \eqref{eq:q-assess}--\eqref{eq:m-assess}) and also
important for the implementation of a simulation algorithm for $\tilde Z$.

The simplest example is the toy example presented in the introduction, i.e.\
the case, where $K$ consists of a single point, $K=\{y_0\}$.
Then, we have $\tilde f(y_0) = f(y_0)$ and, thus, the normalized spectral
representation \eqref{eq:z-tilde} simplifies to 
$\tilde Z(y_0) = \max_{t \in \Pi_0} t$ as $c = \int_{\HH}f(y_0)\,H(\dd f) = 1$.
Thus, only the largest point of $\tilde \Pi$ needs to be considered for a
realization of $\tilde Z(y_0)$ as discussed in the introduction. Next, we deal
with more sophisticated examples.

\subsection{Mixed moving maxima} \label{subsec:mmm}

Let $Z$ be a mixed moving maxima process on some compact set $K \subset \RR^d$,
that is, $f(y) = h(y-x)$ for some random shift $x \in S \subset \RR^d$ and a
random function $h: \RR^d \to [0,\infty)$ and
\begin{equation} \label{eq:h-mmm}
H(C) = (\Lambda \times \pi)\{ (x,h): h(\cdot-x) \in C\},
\end{equation} 
$C \subset \HH \subset [0,\infty)^K$, where $\Lambda$ is locally finite measure
on $S$ and $\pi$ is a probability measure on some Polish space 
$P \subset [0,\infty)^{\RR^d}$. Then, the law $g^*H$ of
$F_t = h_t(\cdot-X_t)$ in \eqref{eq:z-tilde} can be decomposed in the following
way: First, we consider $h_t$ with distribution 
$\PP(h_t \in \dd h) = \xi(h) \pi(\dd h)$, $h \in P$, where
$\xi(h) = c^{-1} \int_S \sup_{y \in K} h(y-z) \Lambda(\dd z)$ and
$c= \int_P \int_S \sup_{y \in K} g(y-z) \, \Lambda(\dd z) \, \pi(\dd g)$. 
Then, $X_t \mid h_t$ follows the law 
$\PP(X_t \in \dd x \mid h_t \in \dd h) = \mu(x,h)$ where 
$$\mu(x,h) = \left(\int_S \sup_{y \in K} h(y-z) \Lambda(\dd z)\right)^{-1} \sup_{y \in K} h(y-x), \quad x \in S, \ h \in P.$$
Here, $t \in \Pi_0$ cannot contribute to $\tilde Z$ if
$$\textstyle t < \inf_{y \in K} \left( \tilde Z(y) \big/ \esssup_{(x,h) \in S\times P} \frac{h(y-x)}{\sup_{\tilde y \in K} h(\tilde y-x)}\right).$$
The right-hand side equals $\inf_{y \in K} \tilde Z(y) / c$ a.s.\ if and only 
if conditions \eqref{eq:cond-sharp-new1} and \eqref{eq:cond-sharp-new2} are 
met. In case of a mixed moving maxima process, \eqref{eq:cond-sharp-new2} is 
equivalent to
\begin{equation} \label{eq:cond-sharp-mmm}
(\Lambda \times \pi)\bigg( (x,h) \in S \times P: \ \frac{h(y-x)}{\sup_{\tilde y \in K} h(\tilde y - x)} > 1 - \varepsilon\bigg) > 0.
\end{equation}

\begin{remark}
 Note that the decomposition of $g^*H$ relies on the fact that $H$ is the
 push forward measure of the product measure $\Lambda \times \pi$. This
 procedure can be generalized for the case that $H$ is the push-forward
 measure of a product measure of the form 
 $\nu_1 \times \ldots \times \nu_n$.
\end{remark}

Note that the results for mixed moving maxima processes can also be applied if
$Z$ is a stationary process with a representation by a stochastic process as in
\eqref{eq:incremental}. In this case,  we may introduce some ``random
shifting''. The following proposition can be shown in exactly the same way as
Theorem 2 in \citet{OKS12}.

\begin{proposition} \label{prop:incr-shift}
 Let $\{W(y), \ y \in \RR^d\}$ be a stochastic process such that the max-stable
 process $\{Z(y), \ y \in \RR^d\}$ given by \eqref{eq:incremental} is 
 stationary. Furthermore, let $S \subset \RR^d$. Then, for any probability
 measure $\Lambda$ on $S$, we have that
 $$ Z(\cdot) =^d \max_{t \in \Pi_0} t W_t(\cdot - X_t),$$
 where $X_t \sim_{i.i.d.} \Lambda$, $t \in (0,\infty)$. Equivalently,
  \begin{equation} \label{eq:incr-shift}
    Z(\cdot) =^d \max_{(t,x,f) \in \Pi} t f(\cdot-x),
  \end{equation}
 where $\Pi$ is a Poisson point process on $(0,\infty) \times S \times P$
 with intensity measure $ t^{-2} \DD t \times \Lambda(\dd x) \times \pi(\dd f)$
 with $\pi$ being the law of $W$ and $P \subset [0,\infty)^{\RR^d}$ being a 
 Polish space.
\end{proposition}

Thus, using representation \eqref{eq:incr-shift} for $Z|_K$, where
$K \subset \RR^d$ is compact, with an arbitrary probability measure
$\Lambda$ on some set $S \subset \RR^d$, we are in the framework of a mixed
moving maxima process. By Proposition \ref{prop:c-finite}, the number $m$ of
considered 
spectral functions in the normalized spectral representation $\tilde Z$ is 
finite a.s.\ if and only if $\EE \sup_{y \in K} W(y) < \infty$. 
Recall that the normalized spectral representation is unique (cf.\ Proposition 
\ref{prop:uniqueness}). Thus, the representation as well as the number $m$
of considered spectral functions do not depend on the choice of $\Lambda$.
However, different choices of $\Lambda$ may lead to different ways of 
decomposing the measure $g^*H$.

\subsection{Monotone, radial symmetric shape function and $K$ a convex, compact set} \label{subsec:radial}

Let $Z$ be a stationary moving maxima process on $\RR^d$ restricted to a convex
compact set with a radial symmetric shape function, that is, let $Z$ be defined
as in Proposition \ref{prop:1} where $H$ is given by
$H(A) = \lambda(\{ x \in \RR^d: \ f_0(\|\cdot - x\|) \in A\})$
for any measurable set $A \subset [0,\infty)^{\RR^d}$, $\lambda$ denotes the
Lebesgue measure on $\RR^d$ and $f_0: [0,\infty) \to [0,\infty)$.
Further, we assume that $f_0$ is monotonically decreasing.
Then, 
$$\textstyle \tilde f_0(x) := \sup_{y \in K} f_0(\|y-x\|) = f_0(d(x,K))$$
with $d(x,K) = \min_{y \in K} \|y-x\|$ and thus $g^*$ as defined in
\eqref{eq:g-star} satisfies
$$ g^*(f_0(\|\cdot - x\|)) = c^{-1} \tilde f_0(x) = c^{-1} f_0(d(x,K)), $$
where $c = \int_{\RR^d} \tilde f_0(x) \DD x = \int_{\RR^d} f_0(d(x,K)) \DD x$.
If $f_0$ is continuous at the origin, then condition 
\eqref{eq:cond-sharp-mmm} is met, which implies 
$Q_{g^*} = \EE\left( c / \inf_{y \in K} \tilde Z(y)\right).$

In the following, we calculate $\tilde f_0(x)$, $x \in \RR^d$, and
$c$ for different cases of $K$.
First, consider the case that $K$ is a $d$-dimensional ball $b(0,R)$ centered
at the origin with radius $R$, i.e.\
$K = b(0,R) = \{ x \in \RR^d: \|x\| \leq R\}$, we get that
$\tilde f_0(x) =  f_0(0) \1_{\|x\|\le R} + \1_{\|x\| > R} f_0(\|x\| - R).$
Assume that the random variable $X$ follows the probability density 
that is proportional to $\tilde f_0$. Then we get
$$\textstyle \PP(\|X\| \le r) = 
 c^{-1} \left( f_0(0) (r\wedge  R)^d + d \int_0^{(r-R) \vee 0} (\tilde r+R)^{d-1} f_0(\tilde r) \DD \tilde r\right)
$$
with $c = f_0(0) R^d + d \int_0^\infty (\tilde r+R)^{d-1} f_0(\tilde r) \DD \tilde r  < \infty$.

Second, consider the case that $K$ is a $d$-dimensional cube is of particular interest,
i.e.\ the case that $K = [-R,R]^d$ for some $R >0$.
Then, we get 
\begin{equation}
\tilde f_0\left((x_1, \ldots, x_d)\right)
={} f_0\left( \|( (|x_1|-R) \vee 0, \ldots, (|x_d|-R) \vee 0 )\| \right). \label{eq:ftilde-cube}
\end{equation}
We consider the subcases $d=1$ and $d=2$ to derive explicit formulae.
If $d=1$, then $K= [-R,R] = b(0,R)$, and, according to the formulae above, we get that
$\tilde f_0(x) = \mathbf{1}_{|x| \leq R} f_0(0) + \mathbf{1}_{|x| > R} f_0(|x|-R)$
and thus, $$c = \int_{\RR} \tilde f_0(x) \dd x = 2Rf_0(0) + \int_{|x| > 0} f_0(|x|) \dd x = 2Rf_0(0) +1.$$
If $d=2$, we obtain
\begin{align*} 
\tilde f_0(x) ={}& \mathbf{1}_{|x_1| \vee |x_2| \leq R} f_0(0) + 2 \cdot \mathbf{1}_{|x_1| \wedge |x_2| \leq R, |x_1| \vee |x_2| > R} f_0((|x_1| \wedge |x_2|)-R) \nonumber\\
& \hspace{0.2cm} + \mathbf{1}_{|x_1| \wedge |x_2| > R} f_0\left(\left\|\left( |x_1|-R, |x_2|-R \right)\right\|\right). 
\end{align*}
Thus,
\begin{align*}
 c ={} & (2R)^2 \cdot f_0(0) + 2 \cdot 2R \cdot \int_\RR f_0(|x|) \DD x + \int_{\RR^2} f_0(\|x\|) \dd x \displaybreak[0]\\
          ={}& 4R^2f(0) + 4R \int_{\RR} f_0(|x|) \DD x + 1.
\end{align*}

Next, we further specify explicit examples on the function $f_0$, under which
the constant $c$ can be further calculated.
\begin{example}\label{ex:mmm}
\begin{enumerate}
\item Indicator function\\ 
   We consider the case that the shape function is the indicator function
   of a ball $b(0,r)$ with radius $r>0$ centered at the origin, i.e. 
   $f_0(\|x\|) = \mathbf{1}_{\|x\| \leq r}$. In this case we have
   $\tilde f_0(x) = \mathbf{1}_{K \oplus b(0,r)}(x)$ and 
   $c = {\rm vol}(K \oplus b(0,r))$ where $\oplus$ denotes morphological
   dilation and ${\rm vol}$ the $d$-dimensional volume. Here, all the 
   finite approximations derived from the normalized spectral representation
   coincide with the corresponding approximations resulting from the algorithm
   proposed by \citet{schlather02}.
\item Smith model\\
   As the second example, we consider the Gaussian extreme value process
   \citep{smith90} where $f_0$ is a Gaussian density function. Here, for 
   simplicity, we assume the shape function to be the density of a 
   multivariate normal random vector 
   $Y \sim \mathcal{N}(\mathbf{0}, \sigma^2 {\rm Id})$ with $\sigma > 0$.
   Thus, it is a radial symmetric monotone function.
   Let $K = [-R,R]^d$ for some $R>0$
   Then, by the considerations above, we get that $\tilde f_0$ is of type
   \eqref{eq:ftilde-cube} and for $d = 1,2$, we obtain
   \begin{align*}
    c = \begin{cases}
         \sqrt\frac{2}{\pi}\frac R \sigma + 1,& d=1\\
         \frac{2}{\pi}\left(\frac R \sigma\right)^2 + 2 \sqrt\frac{2}{\pi}\frac R \sigma + 1, &d=2.
        \end{cases}
   \end{align*}
\end{enumerate}
\end{example}

\begin{remark}
 By the considerations in Subsection \ref{subsec:mmm}, all these results can be
 generalized for the case that the shape function is not deterministic, but
 random with law $\pi$, i.e.
 $$ H(A) = (\lambda \times \pi)(\{ (x,f_0) \in \RR^d \times [0,\infty)^{[0,\infty)}: \ f_0(\|\cdot - x\|) \in A\}).$$
\end{remark}

Now, for random shape functions with law $\pi$, we consider the case that $K$
grows unboundedly. For simplicity, we assume that $K = b(0,R) \subset \RR^d$
with $R \to \infty$. Here, by the considerations above, we have
$\tilde f_0(x) =  f_0(0) \1_{r\le R} + \1_{r > R} f_0(r - R)$ for every
$f_0 \in {\rm supp}(\pi)$.
Then, as a special case of Subsection \ref{subsec:mmm},
we get that $\xi(h) = (\int f_0(0) \pi(\dd f_0))^{-1} (h(0) + o(1))$
and $\mu(x,h) = \frac{R^{-d}}{|b(0,1)|} + o(R^{-d})$.
Thus, we obtain the representation
$$ \tilde Z(y) = \max_{t \in \tilde \Pi_0} t
   \frac{|b(0,R)| \, \int_{h} h(0) \pi(\dd h) \, h_t(y,X_t)}{h_t(0)} + o(1), \quad \|y\| \leq R,$$
where $h_t \sim \xi(h) \pi(\dd h)$ and $X_t \mid h_t \sim \mu(x, h_t) \DD x$,
$t > 0$, are all independent.
   
This representation is very similar to the standard mixed moving maxima 
representation used for the simulation algorithm proposed in 
\citet{schlather02}. The main difference, however, is that the shape functions
$h_t$ are transformed to have the same value at the origin and are drawn with 
modified law $\PP(h_t \in \dd \cdot)$ instead of $\pi$. 
This difference also causes a different asymptotic behavior of the number of
considered shape functions. While
$Q_{g^*} = \EE[(\int f_0(0) \pi(\dd f_0) \cdot |b(0,1)| R^d + o(R^d)) / \inf_{y \in b(0,R)} Z(y)]$, 
the expected number of spectral functions taken into account in Schlather's 
\citeyearpar{schlather02} algorithm equals
$\EE[(\esssup_{f_0 \in {\rm supp}(\pi)} f_0(0) \cdot |b(0,1)| R^d + o(R^d)) / \inf_{y \in b(0,R)} Z(y)]$.
Thus, by using the normalized spectral representation, the number is 
asymptotically decreased by a factor
$\int f_0(0) \pi(\dd f_0) / \esssup_{f \in {\rm supp}(\pi)} f(0).$
For details on Schlather's \citeyearpar{schlather02} algorithm and the number
of considered spectral functions, see Section \ref{sec:simu}.

\subsection{Brown-Resnick Processes} \label{subsec:BR}

We consider a Brown-Resnick process
\begin{equation} \label{eq:br}
 Z(y) = \max_{t \in \Pi_0} t \exp(B_t(y) - \sigma^2(y)/2), \quad y \in K,
\end{equation}
on a compact set $K \subset \RR^d$, where $\Pi_0$ is a Poisson point process
on $(0,\infty)$ with intensity measure $t^{-2} \DD t$ and $B_t$, $t > 0$, are
independent copies of a stochastic process $\{B(y), \ y \in K\}$. Here, $B$
is a zero-mean Gaussian process with stationary increments, variogram $\gamma$,
and variance $\sigma^2(\cdot)$. Note that $Z$ is stationary and its law only
depends on $\gamma$  \citep[cf.][]{KSH09}.
As the representation \eqref{eq:br} is of type \eqref{eq:incremental}, we can 
use the fourth condition of Proposition \ref{prop:c-finite} for the existence 
of the  normalized spectral representation. Thus, the number $m$ from the 
stopping rule \eqref{eq:stopping-trafo} is finite if and only if 
$\EE\left( \sup_{y \in K} \exp\left(B(y) - \sigma^2(y) / 2\right)\right) < \infty.$

However, if $\gamma$ tends to infinity fast enough, the original definition 
turns out to provide inappropriate finite approximations and the mixed moving
maxima representation is a promising option \citep[cf.][]{OKS12}.
Thus, we aim to derive the normalized spectral representation starting
with a stationary mixed moving maxima representation, i.e.\ $H$ is defined by
\eqref{eq:h-mmm}, where $\Lambda$ is the Lebesgue measure on $S=\RR^d$.
By \cite{KSH09}, such a representation exists if 
$B(y) - \sigma^2(y) / 2 \to -\infty$ a.s.\ for $\|y\| \to \infty$
and $B$ has continuous sample paths. In this case, the random variables 
$\tau = \arg\max\left( B(\cdot) - \sigma^2(\cdot)/2\right)$ and
$\upsilon = \max \exp\left( B(\cdot) - \sigma^2(\cdot)/2\right)$ are 
well-defined and, by \cite{EMOS14}, the shape function $h \sim \pi$
is given by 
$$h(\cdot) \stackrel{d}{=} \bigg(\int_{\RR^d} \int_{C(\RR^d)} f(t) \, \PP_{h_0}(\dd f) \DD t \bigg)^{-1} h_0(\cdot)$$
where $h_0$ has the law 
$$ \PP_{h_0}(A) = \frac{\int_0^\infty y \PP\left(\upsilon^{-1} W(\cdot+\tau) \in A,
                       \ \tau \in [0,1]^d \mid \upsilon = y\right) \PP_\upsilon(\dd y)}
                       {\int_0^\infty y \PP(\tau \in [0,1]^d \mid \upsilon = y) \PP_\upsilon(\dd y)},$$
with $W(\cdot) = \exp(B(\cdot) - \sigma^2(\cdot)/2)$.
Thus, we have $\arg\max h = \mathbf{0}$ and 
$\max h = \big(\int_{\RR^d} \int_{C(\RR^d)} f(t) \, \PP_{h_0}(\dd f) \DD t \big)^{-1}$ a.s.

Furthermore, as $B$ has continuous sample paths, we have that, for any compact
set $K \subset \RR^d$, 
$\PP( \sup_{y \in K} (B(y) - \sigma^2(y)/2) < \infty) = 1$ and thus, by 
Theorem 2.1.2 in \cite{adlertaylor},
$$\textstyle \EE \sup_{y \in K} \exp\left(B(y) - \sigma^2(y)/2\right) < \infty.$$ 
Hence, by Proposition \ref{prop:c-finite}, we obtain $c < \infty$, i.e.\
the existence of the normalized spectral representation.
As $S=\RR^d$ and $h$ is continuous at the origin, we get that
\eqref{eq:cond-sharp-mmm} holds, and thus, by the stopping rule 
\eqref{eq:stopping-trafo}, a point $t \in \Pi_0$ cannot contribute to 
$\{\tilde Z(y), \, y \in K\}$ if $t < c^{-1} \inf_{y \in K} \tilde Z(y)$.
Hence, we have a valid stopping rule for Brown-Resnick processes,
as $\inf_{y \in K} \tilde Z(y) > 0$ a.s.

Remind that the normalized spectral functions 
$F_t^* =c F_t / \sup_{y \in K} F_t(y)$ in \eqref{eq:z-tilde} 
are uniquely determined by \eqref{eq:sup-const} with
$$\textstyle c = \EE \sup_{y \in K} \exp\left(B(y) - \sigma^2(y)/2\right) 
    = \int \int_{\RR^d} \sup_{y \in K} f(y-x) \DD x \, \pi(\dd f),$$
(cf.\ Proposition \ref{prop:uniqueness}). 
In particular, the normalized spectral representation for the representation
\eqref{eq:br} is the same as for the equivalent mixed moving maxima 
representation. However, the representations provide different ways to 
decompose the distribution $g^*H$ of $F_t$.

\section{Comparison to the algorithm proposed in \citet{schlather02}} \label{sec:simu}

In this section, we compare the number of spectral functions considered in the
normalized spectral representation to that considered in Schlather's
\citeyearpar{schlather02} algorithm for mixed moving maxima processes.
First, we present the algorithm proposed by \cite{schlather02} and calculate
the number of considered shape functions in the general case. In Subsections
\ref{subsec:comp-theo} and \ref{subsec:comp-simu}, we compare this number to
the corresponding number for the normalized spectral representation in case of
the Smith process \citep{smith90} theoretically and in a simulation study.

Let $\{Z(y): \, y \in \RR^d\}$ be a stationary mixed moving maxima process,
i.e.\ $H$ is given by \eqref{eq:h-mmm} and $\Lambda$ is the Lebesgue measure
on $\RR^d$. In \citet{schlather02}, a simulation algorithm is proposed which is
shown to be exact if the shape functions  $h \in {\rm supp}(\pi)$ are 
jointly bounded and have joint support, i.e.\
$\pi(\{h: \, h(x) < C \text{ for all } x \in \RR^d\}) = 1$ for some $C>0$ and
$\pi(\{h: \, {\rm supp}(h) \subset b(0,r)\}) = 1$ for some $r>0$ 
\citep[][Thm. 4]{schlather02}.
In this case,
\begin{align*}
 Z(y) =_d{}& |K \oplus b(0,r)| \cdot \max_{1 \leq n \leq M} \frac{F_n(y - U_n)}{\sum_{k=1}^n \xi_k},
 \quad y \in K,
\end{align*}
where $\xi_k$ are independent and identically distributed random variables with
standard exponential distribution, $F_k$ follow the law $\pi$, $U_k$ are
uniformly distributed on $K \oplus b(0,r)$ and all these random variables are
independent. Further, $M$ is a random number defined by
$$ M = \min\left\{ m \in \NN: \ \frac{C}{\sum_{k=1}^{m+1} \xi_k} 
                   \leq \inf_{x \in K} \max_{1 \leq n \leq m} \frac{F_n(x-U_n)}{\sum_{k=1}^n \xi_k}\right\}.$$
Here, analogously to Proposition \ref{prop:refinement}, the following result
can be shown.
\begin{proposition} \label{prop:mmm-number-exact}
 The expectation of $M$, defined as above, equals
 $$ \EE M = \EE \left( \frac{|K \oplus b(0,r)| \cdot C}{\inf_{y \in K} Z(y)}\right).$$
\end{proposition}
If the shape functions are not jointly compactly supported, the max-stable 
process $Z$ can be approximated using shape functions which are cut off outside
a compact set $J$, i.e.\ $\tilde F_n(x) = F_n(x) \cdot \mathbf 1_{x \in J}$.
Then, with $\tilde U_k \sim_{i.i.d.} {\rm Unif}(K \oplus \check J)$ where
$\check J = \{-x: \ x \in J\}$, for the process $Z_J(\cdot)$ defined by
$$ Z_J(y) = |K \oplus \check J| \cdot
\max_{n \in \NN} \frac{\tilde F_n(y - \tilde U_n)}{\sum_{k=1}^n \xi_k},
\quad y \in K,$$
the number $M$ of shape functions that need to be considered for the exact
process $\{Z_J(y), \ y \in K\}$ is finite a.s., and, by Proposition 
\ref{prop:mmm-number-exact}, its expectation equals 
$\EE \left( \frac{|K \oplus \check J| \cdot C}{\inf_{y \in K} Z_J(y)}\right)$.

\subsection{Theoretical Comparison in the Case of the Smith Process} \label{subsec:comp-theo}

In order to compare the aforementioned two numbers of spectral functions, we 
consider the Smith process described in Example \ref{ex:mmm} on a rectangle
$[-R,R]^d$ for $d=1,2$. By Example \ref{ex:mmm} and Equations 
\eqref{eq:q-assess}-\eqref{eq:cond-sharp-new2}, the expected number of spectral
functions considered in the normalized spectral representation $\tilde Z$ 
equals
\begin{align*}
 Q_{g^*} ={}& 
 \begin{cases}
 \left(\sqrt\frac{2}{\pi}\frac R \sigma + 1\right) \EE\left(\sup_{y \in [-R,R]} Z(y)^{-1}\right),& d=1\\
 \left(\frac{2}{\pi}\left(\frac R \sigma\right)^2 + 2 \sqrt\frac{2}{\pi}\frac R \sigma + 1\right)
  \EE\left(\sup_{y \in [-R,R]^2} Z(y)^{-1}\right), &d=2.
 \end{cases}
\end{align*}
For the simulation algorithm of \citet{schlather02}, we need an approximation
as described above. Here, a natural choice for cutting off the shape function
is $L = [-k \sigma, k \sigma]^d$ for some $k \in \NN$. Then, the expected
number of considered shape functions equals
$$ \EE M_k = \sqrt{\frac 2\pi}^d \left(\frac{R}{\sigma} + k\right)^d \EE\bigg(\sup_{y \in [-R,R]^d} Z_{[-k\sigma,k\sigma]^d}(y)^{-1}\bigg).$$

Thus, the ratio between the two expected numbers of considered spectral 
functions, $Q_{g^*} / \EE M_k$, can be written as a product
\begin{align} \label{eq:prod}
  \frac{Q_{g^*}}{\EE M_k} = A_{R,k}P_{R,k}
\end{align}
where
\begin{align*}
 A_{R,k} = \begin{cases}
            \frac{R + \sqrt{\pi/2} \sigma}{R + k \sigma}, & d=1,\\ 
            \frac{R^2 + \sqrt{2\pi} \sigma R + \frac \pi 2 \sigma^2}{R^2 + 2k\sigma R + k^2 \sigma^2}, & d=2
           \end{cases} 
\end{align*}
and
$$ P_{R,k} = \frac{\EE\left(\sup_{y \in [-R,R]^d} \tilde Z(y)^{-1}\right)}
     {\EE\left(\sup_{y \in [-R,R]^d} Z_{[-k\sigma,k\sigma]^d}(y)^{-1}\right)}.$$
As we have $Z_{[-k\sigma,k\sigma]^d} \rightarrow_d \tilde Z$ as $k \to \infty$, the
relative number of considered spectral functions asymptotically equals
$Q_{g^*} / \EE M_k = A_{R,k} (1 + o(1))$ as $k \to \infty$.
Note that $A_{R,k} < 1$ if and only if $k > \sqrt{\frac \pi 2}$.
In addition, as $Z_{[-k\sigma,k\sigma]^d}$ is constructed via the cut off shape
functions $\tilde F_n(\cdot) \leq F_n(\cdot)$, we have that $P_{R,k} \leq 1$.
Thus, in the product \eqref{eq:prod}, the first factor $A_{R,k}$ basically 
refers to the area to which the points of the Poisson point process belong, and
the second factor $P_{R,k}$ refers to the exactness of the approximation by
Schlather's \citeyearpar{schlather02} algorithm.

\subsection{Simulation Study for the Smith Process} \label{subsec:comp-simu}

We will now verify the theoretical considerations above in a simulation study.
To this end, for $\sigma=1$, we simulate $Z$ and $Z_{[-k,k]^d}$ for $k=2,3$
on a grid $K=\{-R, -R+h, \ldots, R-h, R\}^d$, $d=1,2$. The density function
$g^*$, however, is chosen as if $K$ was the rectangle $[-R,R]^d$.

In the case $d=1$, for $h=0.1$ and $R \in \{1,2,5,10,50,100\}$ we simulate each
process $N=5000$ times. The values of $Q_{g^*}$ and $\EE M_k$ are 
estimated via the corresponding empirical means denoted by 
$\hat Q_{g^*}$ and $\widehat{\EE M_k}$. For estimation of $P_{R,k}$ we use the
plug-in estimator $\hat P_{R,k}$ based on the empirical means of 
$\sup_{y \in K} \tilde Z(y)^{-1}$ and $\sup_{y \in K} Z_{[-k,k]^d}(y)^{-1}$.
The results of the simulation study are shown in Table \ref{tab:1d}.

\begin{table}
\caption{Results for simulations of $\tilde Z$ and $Z_{[-k,k]}$, $k=2,3$, on
$\{-R, -R+0.1, \ldots, R-0.1, R\}$ for different $R$. For each case, $A_{R,k}$
and the estimates for $Q_{g^*}$, $\EE M_k$ and $P_{R,k}$ as defined in 
Subsection \ref{subsec:comp-theo} are displayed, based on $N=5000$ simulations
of each process.}
\label{tab:1d}       
\begin{tabular}{rrcrrrrcrrrr}
\hline\noalign{\smallskip}
$R$ & $\hat Q_{g^*}$ 
& \ & $\widehat{\EE M_2}$ & $\frac{\hat \Qorig_{g^*}}{\widehat{\EE M_2}}$ & $A_{R,2}$ & $\hat P_{R,2}$ 
& \ & $\widehat{\EE M_3}$ & $\frac{\hat \Qorig_{g^**}}{\widehat{\EE M_3}}$ & $A_{R,3}$ & $\hat P_{R,3}$ \\
\noalign{\smallskip}\hline\noalign{\smallskip}
  1 &   3.12 &&   4.38 & 0.71 & 0.75 & 0.94 &&   5.46 & 0.57 & 0.56 & 1.00 \\
  2 &   5.73 &&   7.57 & 0.76 & 0.81 & 0.94 &&   8.93 & 0.64 & 0.65 & 0.98 \\
  5 &  15.82 &&  18.82 & 0.84 & 0.89 & 0.95 &&  19.98 & 0.79 & 0.78 & 1.02 \\
 10 &  35.63 &&  40.57 & 0.88 & 0.94 & 0.94 &&  41.16 & 0.87 & 0.87 & 1.00 \\
 50 & 239.75 && 257.61 & 0.93 & 0.99 & 0.94 && 247.35 & 0.97 & 0.97 & 1.00 \\
100 & 540.44 && 579.11 & 0.93 & 0.99 & 0.94 && 550.70 & 0.98 & 0.98 & 1.00 \\
\noalign{\smallskip}\hline
\end{tabular}
\end{table}

First, we note that -- in accordance to Equation \eqref{eq:prod} --
$Q_{g^*}$ is always smaller than $\EE M_k$. For instance, for $R=1$,
the number of considered shape functions is decreased by $29\%$ ($k=2$) and 
$43\%$ ($k=3$), respectively. Furthermore, we observe that $P_{R,k}$ seems to
be almost constant in $R$, namely $P_{R,2} \approx 0.95$ and 
$P_{R,3} \approx 1$ which shows that the approximation of $\tilde Z$ by $Z_{[-3,3]}$
is largely good for $h=0.1$. Thus, the behavior of $Q_{g^*} / \EE M_k$ is 
basically driven by $A_{R,k}$ which tends to $1$ as $R \to \infty$. For large
$R$, $Q_{g^*} / \EE M_k \approx P_{R,k}$. Thus, we get the surprising
fact that  $\EE M_2 > \EE M_3$ even though the approximation of $\tilde Z$ by
$Z_{[-2,2]}$ is less accurate than by $Z_{[-3,3]}$.

\begin{table}
\caption{Results for simulations of $\tilde Z$ and $Z_{[-k,k]^2}$, $k=2,3$, on
$\{-R, -R+0.25, \ldots, R-0.25, R\}^2$ for different $R$. For each case, 
$A_{R,k}$ and the estimates for $Q_{g^*}$, $\EE M_k$ and $P_{R,k}$ as defined
in Subsection \ref{subsec:comp-theo} are displayed, based on $N=2500$ simulations
of each process.}
\label{tab:2d}       
\begin{tabular}{rrcrrrrcrrrr}
\hline\noalign{\smallskip}
$R$ & $\hat Q_{g^*}$ 
& \ & $\hat M_2$ & $\frac{\hat \Qorig_{g^*}}{\widehat{\EE M_2}}$ & $A_{R,2}$ & $\hat P_{R,2}$ 
& \ & $\hat M_3$ & $\frac{\hat \Qorig_{g^*}}{\widehat{\EE M_3}}$ & $A_{R,3}$ & $\hat P_{R,3}$ \\
\noalign{\smallskip}\hline\noalign{\smallskip}
  1 &   8.14 &&  14.86 & 0.55 & 0.56 & 0.96 &&  26.37 & 0.31 & 0.32 & 0.96 \\
  2 &  26.32 &&  40.17 & 0.66 & 0.66 & 1.00 &&  61.07 & 0.43 & 0.42 & 1.03 \\
  5 & 150.89 && 189.83 & 0.79 & 0.80 & 0.99 && 247.10 & 0.61 & 0.61 & 1.00 \\
 10 & 636.03 && 727.33 & 0.87 & 0.88 & 0.99 && 839.55 & 0.76 & 0.75 & 1.01 \\
\noalign{\smallskip}\hline
\end{tabular}
\end{table}

For $d=2$, $R \in \{1,2,5,10\}$ and $h=0.25$, each process is simulated 
$N=2500$ times. The results are shown in Table \ref{tab:2d}. In general,
the results are similar to our observations for $d=1$. However, for $d=2$ the
improvements compared to Schlather's \citeyearpar{schlather02} algorithm are 
even more distinct. In the case $R=1$, the number of considered spectral 
functions is decreased by $45\%$ ($k=2$) and $69\%$ ($k=3$), respectively.
However, the results of the algorithm by \citet{schlather02} seem to be 
quite accurate even for $k=2$ as $P_{R,k}$ suggests.

\section{Summary and Discussion} \label{sec:discussion}

Whilst in the definition of a max-stable process an infinite number of spectral
functions is involved, the minimal number of spectral functions that are
actually to be considered in a simulation is an open problem.
We consider two substitution problems, problems \eqref{eq:simple} and
\eqref{eq:refined}, and show that the unique normalized spectral representation
is a solution in both cases. Although we feel that problem \eqref{eq:refined}
is rather close to the original problem \eqref{eq:original}, it remains unclear
whether the normalized spectral representation is also the solution to the
original one. It is even not known whether different initial choices of the
spectral representation in \eqref{eq:def} may lead to the same solution via
renormalizations $g$ in \eqref{eq:0} and whether the solution is unique. This
is left for future research. 

Section \ref{sec:simu} reveals two remarkable facts: (i) the potential of the
approach based on the normalized spectral representation to improve the 
algorithm of \cite{schlather02} and (ii) the occasional occurrence of a smaller
number of considered shape functions in a better approximation. Neither a
careful coding that exploits our fundamental results seems to be straightforward
nor are the implications on the real running times foreseeable.
This is also left for future research.


\bibliography{spectral.bib}
\bibliographystyle{imsart-nameyear}

\end{document}